\documentclass[10.9pt,a4paper]{amsart}
\usepackage{a4wide}
\usepackage[utf8]{inputenc}
\usepackage[english]{babel}

\usepackage{amsfonts,amssymb,amsmath}

\usepackage[nobysame,alphabetic, initials]{amsrefs}

\usepackage{comment}

\usepackage{graphicx}

\usepackage[dvipsnames]{xcolor}
\usepackage[colorlinks,
linkcolor={blue!50!black},
citecolor={blue!50!black},
urlcolor={black!80!black}]{hyperref}
\usepackage{enumerate}

\usepackage{url}

\usepackage[capitalize]{cleveref}

\usepackage{tikz}
\usepackage{tikz-cd}
\usepackage[all]{xy}
\usepackage{enumitem}
\makeatletter
\newcommand{\mylabel}[2]{#2\def\@currentlabel{#2}\label{#1}}
\usetikzlibrary{calc, matrix, arrows, cd}

\newcommand{\bsm}{\left(\begin{smallmatrix}}
	\newcommand{\esm}{\end{smallmatrix}\right)}
\newtheorem{innercustomthm}{Theorem}
\newenvironment{customthm}[1]
{\renewcommand\theinnercustomthm{#1}\innercustomthm}
{\endinnercustomthm}

\numberwithin{equation}{section}

\newtheorem{theorem}[equation]{Theorem}

\newtheorem{corollary}[equation]{Corollary}
\newtheorem{lemma}[equation]{Lemma}
\newtheorem{proposition}[equation]{Proposition}

\theoremstyle{definition}
\newtheorem{definition}[equation]{Definition}

\newtheorem{remark}[equation]{Remark}

\newtheorem{construction}[equation]{Construction}
\newtheorem{notation}[equation]{Notation}
\newtheorem*{claim}{Claim}
\newtheorem*{claim*}{Claim}

\AddToHook{env/theorem/begin}{\crefalias{equation}{theorem}}
\AddToHook{env/convention/begin}{\crefalias{equation}{convention}}
\AddToHook{env/notation/begin}{\crefalias{equation}{notation}}
\AddToHook{env/proposition/begin}{\crefalias{equation}{proposition}}
\AddToHook{env/corollary/begin}{\crefalias{equation}{corollary}}
\AddToHook{env/definition/begin}{\crefalias{equation}{definition}}
\AddToHook{env/lemma/begin}{\crefalias{equation}{lemma}}
\AddToHook{env/question/begin}{\crefalias{equation}{question}}
\AddToHook{env/example/begin}{\crefalias{equation}{example}}
\AddToHook{env/conjecture/begin}{\crefalias{equation}{conjecture}}
\AddToHook{env/remark/begin}{\crefalias{equation}{remark}}
\AddToHook{env/construction/begin}{\crefalias{equation}{const}}

\crefname{notation}{Notation}{Notation}
\crefname{convention}{Convention}{Conventions}
\crefname{theorem}{Theorem}{Theorems}
\crefname{proposition}{Proposition}{Propositions}
\crefname{corollary}{Corollary}{Corollaries}
\crefname{definition}{Definition}{Definitions}
\crefname{lemma}{Lemma}{Lemmas}
\crefname{question}{Question}{Questions}
\crefname{example}{Example}{Examples}
\crefname{conjecture}{Conjecture}{Conjectures}
\crefname{remark}{Remark}{Remarks}
\crefname{const}{Construction}{Constructions}

\newcommand{\Z}{\mathbb{Z}}

\newcommand{\Hom}{\operatorname{Hom}}

\newcommand{\proj}{\operatorname{proj}}

\newcommand{\aug}{\operatorname{aug}}
\newcommand{\incl}{\operatorname{incl}}

\newcommand{\id}{\operatorname{id}}

\newcommand{\ev}{\operatorname{ev}}

\newcommand{\im}{\operatorname{im}}
\newcommand{\wt}{\widetilde}
\newcommand{\wh}{\widehat}
\DeclareMathOperator{\Res}{Res}

\definecolor{bettergreen}{rgb}{0.0, 0.5 0.0}

\begin{document}
	\title{The first relative~$k$-invariant}
	
	\author[A.~Conway]{Anthony Conway}
	\address{The University of Texas at Austin, Austin TX}
	\email{anthony.conway@austin.utexas.edu}
	\author[D.~Kasprowski]{Daniel Kasprowski}
	\address{University of Southampton, United Kingdom}
	\email{d.kasprowski@soton.ac.uk }
	
	\maketitle
	
\begin{abstract}
Motivated by work on the homotopy classification of~$4$-manifolds with boundary,  we define a relative~$k$-invariant for pairs of spaces that are homotopy equivalent to CW pairs.
We show that for such a pair $(X,Y)$ with Postnikov $2$-type $X \to P_2(X)$, the relative $k$-invariant is the obstruction to the existence of a section~$B\pi_1(X)\to P_2(X)$ extending~$Y \hookrightarrow X \to P_2(X)$.
Given CW pairs~$(X_0,Y_0)$ and~$(X_1,Y_1)$,  as well as a map~$h \colon Y_0 \to Y_1$, we also prove that relative~$k$-invariants provide a complete obstruction to constructing a map~$X_0^{(3)} \cup Y_0 \to X_1$ that extends~$h$ and induces given isomorphisms on~$\pi_1$ and~$\pi_2$.
\end{abstract}	
	
\section{Introduction}

Given a space~$X$ with Postnikov tower~$\ldots \to P_2(X) \to B\pi_1(X)$,  the (first) \emph{$k$-invariant of~$X$} is the obstruction~$k_X \in H^3(B\pi_1(X);\pi_2(X))$ to finding a section of~$P_2(X) \to B\pi_1(X)$.
This invariant,  first introduced in~\cite{EilenbergMacLane-operators}, is a mainstay in algebraic topology, but 
also plays a role in geometric topology and, more specifically,  in the (homotopy) classification of manifolds.
For example,  the homotopy type of a closed oriented~$4$-manifold~$X$ with finite fundamental group is often determined by its quadratic~$2$-type~$(\pi_1(X),\pi_2(X),k_X,\lambda_X)$, where~$\lambda_X$ denotes the equivariant intersection form of~$X$; see e.g.~\cite{HambletonKreck}.
Motivated by upcoming work on the homotopy classification of~$4$-manifolds with nonempty boundary~\cite{ConwayKasprowski4Manifolds}, the present article is concerned with relative~$k$-invariants and collects several results on the topic that might be known to experts but that we have not been able to find in the literature.

\subsection*{Statement of results}

Let~$(X,Y)$ be a pair of spaces that is homotopy equivalent to a~CW pair,  let $c \colon X \to P_2(X)$ be the Postnikov $2$-type of $X$, and given a map~$\nu \colon X \to B\pi_1(X)$,  we write~$M(\nu|_Y)$ for the mapping cylinder of~$\nu|_Y$.
The relative~$k$-invariant refers to a cohomology~class
$$k_{X,Y}^\nu\in H^3(M(\nu|_Y),Y;\pi_2(X)).$$
The definition of this class is described precisely in Section~\ref{sec:kInvariantNotCW}, but for the moment we note that~$k_{X,Y}^\nu$ is the obstruction to the existence of a dashed map
making the following diagram commute:
\[\begin{tikzcd}
	Y\ar[r,"c|_Y"]\ar[d,"\nu|_Y"']&P_2(X)\ar[d]\\
	B\pi_1(X)\ar[r,"\id"]\ar[ur,dashed]&B\pi_1(X).
\end{tikzcd}\]
Up to isomorphism, $k_{X,Y}^\nu$ is independent of the choice of the map~$\nu$; see Remark~\ref{rem:Indepj}.
The reason for which we write $H^3(M(\nu|_Y),Y;\pi_2(X))$ instead of $H^3(B\pi_1(X),Y;\pi_2(X))$ is to make it apparent that we choose a model of~$B\pi_1(X)$ that contains $Y$ as a subspace.
Also, $k_{X,\emptyset}^\nu$ agrees with~$k_X$.

The following result 
(the absolute version of which can be found in~\cite{EilenbergMacLane-operators,MacLaneWhitehead})
plays a crucial role in upcoming work on the homotopy classification of~$4$-manifolds with boundary~\cite{ConwayKasprowski4Manifolds}. 

\begin{theorem}
	\label{thm:RealiseAlgebraic3Type-intro}
	Let~$(X_0,Y_0)$ and~$(X_1,Y_1)$ be pairs of spaces that are homotopy equivalent to~CW pairs, and let~$c_1\colon X_1\to P_2(X_1)$ be the Postnikov~$2$-type of~$X_1$.
For a map~$h \colon Y_0 \to Y_1$, an isomorphism~$u\colon \pi_1(X_0)\to \pi_1(X_1)$ with~$u\circ(\iota_0)_*= (\iota_1)_* \circ h$, and an~$u$-equivariant homomorphism~$F \colon \pi_2(X_0) \to~\pi_2(X_1)$, the following assertions are equivalent:
	\begin{itemize}
		\item there is a map~$c_0 \colon X_0\to P_2(X_1)$ such that 
		\[(c_0)_*=u \colon \pi_1(X_0)\to \pi_1(P_2(X_1)) \cong \pi_1(X_1),\quad c_1 \circ h\simeq c_0|_{Y_0},  \quad \text{and}\]
		\[ 
		(c_0)_*=F\colon \pi_2(X_0)\to \pi_2(P_2(X_1)) \xleftarrow{(c_1)_*,\cong}\pi_2(X_1);
		\]
		\item the relative~$k$-invariants satisfy
		$$(u,h)^*(k_{X_1,Y_1}^{\nu_1})=F_*(k_{X_0,Y_0}^{\nu_0}) \in H^3(M(\nu_0|_{Y_0}),Y_0;\Res_u \pi_2(X_1))$$
for every~$\nu_i \colon X_i \to B\pi_1(X_i)$ that induces the identity on $\pi_1$ for~$i=0,1$. 
The map~$(u,h)^*$ is defined in \cref{lem:InducedMapkInvariantsNotCW}.
	\end{itemize}
	\end{theorem}
Here,  a homomorphism~$F \colon \pi_2(X_0) \to \pi_2(X_1)$ is called \emph{$u$-equivariant} if~$F(\gamma x)=u(\gamma)F(x)$ for every~$x \in \pi_2(X_0)$ and~$\gamma \in \pi_1(X_0)$.
Equivalently,~$F$ is a~$\Z[\pi_1(X_0)]$-linear map~$\pi_2(X_0) \to \Res_u \pi_2(X_1)$, where given a~$\Z[\pi_1(X_1)]$-module~$H$, we write~$\Res_u H$ for the~$\Z[\pi_1(X_0)]$-module whose underlying abelian group is~$H$,  but with~$\pi_1(X_0)$-action given by~$\gamma \cdot x=u(\gamma)x$.

\begin{remark}
The reason we work with spaces that are merely homotopy equivalent to CW pairs (instead of working with CW pairs as is common in the literature) is to ensure that our results apply to~$4$-dimensional topological manifolds.
\end{remark}	
	
If~$(X_0,Y_0)$ and~$(X_1,Y_1)$ are CW pairs,  Theorem~\ref{thm:RealiseAlgebraic3Type-intro} admits  the following reformulation.

\begin{corollary}
	Let~$(X_0,Y_0)$ and~$(X_1,Y_1)$ be CW pairs,  and let~$\iota_j \colon Y_j \to X_j$ be the inclusion for~$j=0,1$.
For a group isomorphism~$u\colon \pi_1(X_0)\to \pi_1(X_1)$, a map~$h \colon Y_0 \to Y_1$ satisfying~$u\circ(\iota_0)_*= (\iota_1)_* \circ h$, and an~$u$-equivariant homomorphism~$F \colon \pi_2(X_0) \to \pi_2(X_1)$, the following assertions are equivalent:	
	\begin{itemize}
		\item there is a map~$f\colon X_0^{(3)}\cup Y_0\to X_1$  that extends~$h$ and induces~$u$ and~$F$;
		\item the relative~$k$-invariants satisfy
		$$(u,h)^*(k_{X_1,Y_1}^{\nu_1})=F_*(k_{X_0,Y_0}^{\nu_0}) \in H^3(M(\nu_0|_{Y_0}),Y_0;\pi_2(X_1))$$
for every~$\nu_i \colon X_i \to B\pi_1(X_i)$ that induces the identity on $\pi_1$ for~$i=0,1$.
	\end{itemize}
\end{corollary}
\begin{proof}
By \cref{thm:RealiseAlgebraic3Type-intro}, it suffices to show that there exists a map~$f \colon X_0^{(3)}\cup Y_0\to X_1$ as in this corollary if and only if there exists a map~$c_0 \colon X_0 \to P_2(X_1)$ as in \cref{thm:RealiseAlgebraic3Type-intro}. 
In one direction,   starting from~$f$,  since the inclusion~$X_0^{(3)}\cup Y_0\to X_0$ induces
a homotopy equivalence on Postnikov~$2$-types,  the required $c_0$ arises as the composition
	$$ X_0 \to P_2(X_0) \simeq P_2(X_0^{(3)}\cup Y_0) \xrightarrow{P_2(f)} P_2(X_1).$$
	 It remains to construct~$f$ from~$c_0$. 
	 For this, choose a model for~$P_2(X_i)$ that is obtained from~$X_i$ by attaching cells of dimension~$\geq 4$. 
	 Homotoping~$c_0$ if necessary (using cellular approximation), we assume that it is cellular. 
Restricting this map to~$X_0^{(3)}$ leads to a map~$f'\colon X_0^{(3)}\to X_1^{(3)}\hookrightarrow X_1$ that induces~$u$ and~$F$ and such that~$f'|_{Y_0^{(3)}}\simeq \iota_{Y_1}\circ h|_{Y_0^{(3)}}$. 
Since~$Y_0^{(3)}\hookrightarrow X_0^{(3)}$,  the homotopy extension property shows that $f'$ is homotopic to a map~$f''\colon X_0^{(3)}\to X_1^{(3)}\hookrightarrow X_1$ that induces~$u$ and~$F$ and such that~$f''|_{Y_0^{(3)}}=\iota_{Y_1}\circ h|_{Y_0^{(3)}}$. 
Using $h$ to extend~$f''$ over the rest of~$Y_0$ yields the required~$f$.
\end{proof}
	
\subsection*{Organisation}
Section~\ref{sub:PrepKInvariant} collects some conventions and lemmas on chain complexes.
Section~\ref{sub:kinvariantChain} introduces the relative~$k$-invariant of a chain map.
Section~\ref{sub:kinvariantCW} focuses on the relative~$k$-invariant of~CW pairs, whereas Section~\ref{sec:kInvariantNotCW} defines the relative~$k$-invariant in general and proves our main results.
Section~\ref{sec:AdditionalProperties} collects some additional properties of relative $k$-invariants.

\subsection*{Acknowledgments}
AC was partially supported by the NSF grant DMS~2303674.

\subsection*{Conventions}
Unless specified otherwise, spaces are assumed to path-connected and equipped with a basepoint.
Maps are assumed to be basepoint preserving and if $(X,Y)$ is a pair, the basepoint is assumed to lie in $Y$; if~$Y$ is empty, we assume that~$x \in X$.
Throughout this paper,  for every space $X$, we choose a model $B\pi_1(X)$ and assume that all maps $X \to B\pi_1(X)$ that we consider induce the identity on fundamental groups.
The mapping cylinder of a map $f \colon X \to Y$ is denoted $M(f)$, whereas the mapping cone of a chain map $g \colon C \to D$ is denoted $\operatorname{Cone}(g)$.
	
\section{Preparatory lemmas}
\label{sub:PrepKInvariant}

This section sets up our conventions on chain complexes and proves some preliminary lemmas needed to define and study relative $k$-invariants.
\medbreak

In what follows,  chain complexes are assumed to be free and have no chain modules in negative dimensions. 
Given a chain complex $K_*$, we write~$B_i(K)\subseteq K_i$ for the boundaries,  and~$Z_i(K)\subseteq K_i$ for the cycles of degree~$i$.
As in~\cite[page 54]{EilenbergMacLane-operators}, an~\emph{augmented}~$\Z[\pi]$-chain complex~$(K_*,\aug)$ is a chain complex $K_*$ together with a surjective map~$\aug  \colon K_0 \to \Z$, such that
$$ \ldots \to K_1 \xrightarrow{d_1} K_0 \xrightarrow{\aug} \Z \to 0$$
is again a~$\Z[\pi]$-chain complex.
In other words, one requires~$\aug(\gamma x)=\aug(x)$ for every~$x \in K_0$ and every~$\gamma \in \pi$ as well as~$\aug \circ d_1=0$. 
We frequently omit the map~$\aug$ from the notation and just say that~$K_*$ is an augmented chain complex.
Additionally, we set~$\widetilde{H}_0(K):=\ker(\aug)/B_0(K)$ and say that~$K_*$ is \emph{acyclic in dimensions~$<q$} if~$\widetilde{H}_0(K)=0$ and~$H_i(K)=0$ for~$0<i<q$.
Finally,  our notation for truncations is 
$$ K_*|_{[0,n]}:=\left( 0 \to K_n \to K_{n-1} \ldots \to K_1 \to K_0 \to 0\right).$$
The next lemma is a relative version of the uniquness portion of~\cite[Proposition 5.1]{EilenbergMacLane-operators} which, roughly speaking, states that if $K_*'$ is a chain complex that is acyclic in dimensions~$<q$,  then for any chain complex $K_*$ there is, up to homotopy, a unique 
chain map $K_*|_{[0,q-1]} \to K'$.

\begin{lemma}
	\label{lem:EilenbergMaclane51}
	Let~$K_*,K_*'$ be augmentable~$\Z[\pi]$-chain complexes with~$K_*'$ acyclic in dimensions~$<q$,  let~$L_* \leq K_*$ and~$L_*' \leq K_*'$ be subcomplexes such that~$L_i \leq K_i$ and~$L_i' \leq K_i'$ are summands for each~$i$.
	If there exists two chain maps~$\alpha,\beta \colon K_* \to K_*'$ 
	\begin{enumerate}
		\item that induce chain maps between the augmented chain complexes,i.e.\ 
		$$\aug \circ \beta_0=\aug=\aug \circ \alpha_0 \colon K_0 \to \Z,$$
		\item whose restrictions~$\alpha|_L,\beta|_L \colon L_* \to L_*'$ are chain homotopic via a chain homotopy
$$\psi:=\{ \psi_i \colon L_i \to L_{i+1}'\}_i  \colon \alpha|_L \simeq \beta|_L,$$
	\end{enumerate}
	then~$\alpha,\beta$ are chain homotopic in dimensions~$\leq q-1$ via a chain homotopy
	$$D:=\{ D_i \colon K_i \to K_{i+1}'\}_{i \leq q-1}$$
	that agrees with~$\psi$ on~$L|_{[0,q-1]}.$
\end{lemma} 
\begin{proof}
	The proof is a small modification of the proof of~\cite[Theorem 5.1]{EilenbergMacLane-operators} and proceeds by induction on~$i$.
	We begin by defining~$D_0 \colon K_0 \to K_1'$ that extends~$\psi_0$.
	Since~$K_*'$ is acyclic, we have~$\widetilde{H}_0(K')=\ker(\aug)/B_0(K')=0$.
	Since~$\aug \circ (\alpha_0-\beta_0)=0$, the map~$\alpha_0-\beta_0 \colon K_0 \to K_0'$ takes values in~$\ker(\aug)=B_0(K')$.
	Lift this map to~$D_0 \colon K_0 \to K_1'$ by defining it on~$L_0$ to be~$\psi_0$ and extending it arbitrarily over the rest of~$C_0$.
	
	The proof of the induction step is now entirely analogous to the argument from~\cite{EilenbergMacLane-operators}.
	The only difference is that when lifting~$E_i:=\alpha_i-\beta_i-D_{i-1}\circ d_i^K \colon C_i \to B_i(K')$ to~$D_i \colon C_i \to B_{i+1}(K)$,  we define~$D_i$ to be~$\psi_i$ on~$L_i \leq K_i$ but pick an arbitrary lift on the remainder of~$K_i$.
\end{proof}

Applying~\cite[Proposition 5.1]{EilenbergMacLane-operators} to the case where~$K_*'$ is acyclic (i.e. $q=\infty)$ shows that,  up to chain homotopy, there is a unique map $K_* \to K_*'$ that preserves augmentations~\cite[Corollary~5.2]{EilenbergMacLane-operators}.
In particular, this applies if $K_*'$ is a free resolution of $\Z$.
\begin{notation}
	For the rest of the section, we fix a free~$\Z[\pi]$-resolution~$C_*^\pi$ of~$\Z$.
	
Additionally,  given a chain map $f_* \colon A_* \to B_*$ and a map $\varrho \colon A_n \to  Z_n(B) \subset B_n$, we write $f + \varrho \colon A_*|_{[0,n]} \to B_*$ for the chain map given by $f_i$ in degree $i \neq n$ and $f_n+\varrho$ in degree $n$.
\end{notation}	

Since~$C_*^\pi$ is an acyclic augmentable~$\Z[\pi]$-chain complex,  up to chain homotopy,  there is a unique map $t \colon K_* \to C_*^\pi$.
The next lemma is a (relative) partial converse, as it builds a map from a subcomplex of~$C_*^\pi$ into $K_*$; this result will be used to define the relative $k$-invariant.

\begin{lemma}
	\label{lem:existence-alpha}
	Let~$K_*$ be an augmented~$\Z[\pi]$-chain complex that is acyclic in dimensions~$<2$, let~$t \colon K_* \to C_*^\pi$ be a chain map that preserves augmentations, let~$i_L\colon L_*\to K_*$ be a chain map,
	and set~$t_L:=t \circ i_L$. 
	There is an augmentation-preserving chain map 
	$$\alpha\colon C_*^\pi|_{[0,2]}\to K_*|_{[0,2]}$$ 
and a map $\phi\colon L_2\to Z_2(K)$ such that~$\alpha\circ t_L$ and~$i_L+\phi$ are chain homotopic as maps~$L_*|_{[0,2]}\to K_*$.
	
	Given an augmentation-preserving chain map~$\alpha\colon C_*^\pi|_{[0,2]}\to K_*|_{[0,2]}$,  such a $\phi$ exists.
\end{lemma}
\begin{proof}
	Since~$K_*$ is acyclic in dimensions~$<2$, there is a free resolution~$F$ of~$\Z$ containing~$K_*$ as a subcomplex with~$F_*|_{[0,2]}=K_*|_{[0,2]}$ (pick a free resolution $F'_*$ of $\ker d_2^K$ and then take $F_i=K_i$ for~$i <3$ and~$F_i=F'_{i-3}$ for~$ i>2$).
	Pick an augmentation-preserving chain homotopy equivalence~$\wh\alpha\colon C_*^\pi\to F$ and define 
	$$\alpha:=\wh\alpha|_{[0,2]} \colon C_*^{\pi}|_{[0,2]} \to F_*|_{[0,2]}=K_*|_{[0,2]}.$$
	It remains to define $\phi$ and verify that~$\alpha\circ t_L$ and~$i_L+\phi$ are chain homotopic as maps~$L_*|_{[0,2]}\to K_*$.
	Let~$\wh\alpha^{-1}$ be a chain homotopy inverse of~$\wh\alpha$ and let~$\iota\colon K_* \to F$ be the inclusion. 
	By uniqueness of~$t$,~$t\simeq \wh\alpha^{-1}\circ\iota$ up to degree~$2$.
	It follows that
	\[\alpha\circ t_L=\wh \alpha\circ t \circ i_L\simeq \wh\alpha \circ \wh\alpha^{-1}\circ\iota\circ i_L\simeq \iota\circ i_L\colon L|_{[0,2]}\to F_*.\]
Let $\{D_i\colon L_i\to F_{i+1}\}_{i=0,1,2}$ be a chain homotopy and define 
	$$\phi:=d_3^F\circ D_2\colon L_2\to\ker(d_2^F)=\ker(d_2^K).$$ 
Since~$F|_{[0,2]}=K|_{[0,2]}$, a chain homotopy~$\alpha\circ t_L\simeq i_L+\phi\colon L_*|_{[0,2]}\to K_*$ can be obtained by taking~$D_i \colon L_i \to F_{i+1}=K_{i+1}$ in degrees $0,1$ and the trivial map $L_2\xrightarrow{0}K_3$ in degree $2$.
	
It remains to prove the last statement.
First, observe that given an augmentation-preserving chain map $\alpha\colon C_*^\pi|_{[0,2]}\to K_*|_{[0,2]}=F_*|_{[0,2]}$,  there is a chain homotopy equivalence~$\wh\alpha\colon C_*^\pi\to F$ with~$\alpha=\wh\alpha|_{[0,2]}$: since $F_*$ is acyclic,  $\alpha$ extends to~an augmentation preserving~$\wh\alpha\colon C_*^\pi\to F$
and since $C_*^\pi$ and $F_*$ are both acylic augmented chain complexes,  this must be a chain homotopy equivalence; see e.g.~\cite[proof of Theorem 5.1]{EilenbergMacLane-operators}.
Following through the remainder of the proof above then leads to the final sentence of the lemma.
\end{proof}

We conclude with a result concerning mapping cones.
Here,  recall that the mapping cone of a chain map $i \colon L_* \to K_*$ is the chain complex with $n$-th chain module~$\operatorname{Cone}(i)_n=K_n \oplus L_{n-1}$ and~$n$-th differential $\bsm d_n^K & i_{n-1} \\ 0 &-d_{n-1}^L\esm$.

\begin{lemma}
\label{lem:Cofibre}
	If a chain complex $K_*'$ is acyclic in degrees $>0$,  then any homotopy commutative diagram of chain maps
	\[\begin{tikzcd}
		L_*\ar[r,"i"]\ar[d,"h"]&K_*\ar[d,"v"]\\
			L'_*\ar[r,"i'"]&K'_*
	\end{tikzcd}
	\]
	induces, up to homotopy, a unique chain map $(h,v)\colon \operatorname{Cone}(i)_*\to \operatorname{Cone}(i')_*$ such that the diagram
		\[\begin{tikzcd}
	K_*\ar[d,"v"]\ar[r]&\operatorname{Cone}(i)_*\ar[d,"{(h,v)}"]\ar[r,"\proj_2"]&L_{*-1}\ar[d,"h"]\\
		K'_*\ar[r]&\operatorname{Cone}(i')_*\ar[r,"\proj_2"]&L'_{*-1}
	\end{tikzcd}
	\]
	is commutative. Moreover,
\begin{enumerate}
	\item if $h\simeq h'$ and $v\simeq v'$, then $(h,v)\simeq (h',v')$,
	\item if $h$ and $v$ are homotopy equivalences, then so is $(h,v)$.
\end{enumerate}
\end{lemma}
\begin{proof}
The map given by~$\bsm v&\psi\\ 0 &h\esm$, where~$\psi$ is any homotopy~$v\circ i\simeq i' \circ h$, makes the lower diagram commute.
 Conversely, any map that makes this diagram commute must have the form~$\bsm v&\psi\\ 0 &h\esm$ with~$\psi$ a homotopy~$v\circ i\simeq i' \circ h$.  Since~$K'_*$ is acyclic in degrees~$>0$, any two homotopies~$v\circ i\simeq i' \circ h$ are homotopic~\cite[Corollary 5.2]{EilenbergMacLane-operators}, 
and hence the homotopy commutative square defines a map~$(h,v):=\bsm v&\psi\\ 0 &h\esm$ unique up to homotopy. 
The remaining verifications are left to the reader.
		\end{proof}

\section{Relative $k$-invariants of chain complexes.}
\label{sub:kinvariantChain}

This section defines the $k$-invariant associated to an augmented chain complex $K_*$ and a chain map $i_L \colon L_* \to K_*$.
Our approach follows that of Eilenberg-Maclane~\cite{EilenbergMacLane-operators} which treats the absolute case. 
For the remainder of the section,  we continue to fix a free resolution~$C_*^\pi$ of~$\Z$.

\begin{construction}
	\label{cons:Relativek}
	Let~$K_*$ be an augmented chain complex that is acyclic in dimensions~$<2$,  let~$i_L\colon L_*\to K_*$ be a chain map,  let~$t \colon K_* \to C_*^\pi$ be a chain map 
	that preserves augmentations, set~$t_L:=t \circ i_L$,  let~$\operatorname{Cone}(t_L)_*$ be the mapping cone of~$t_L$, and let~$\alpha\colon C_*^\pi|_{[0,2]}\to K_*|_{[0,2]}$ and let~$\phi\colon L_2\to Z_2(K)$ be as in Lemma~\ref{lem:existence-alpha}. 
	Consider the cochain
	\[\theta_{K,L}^t \colon \operatorname{Cone}(t_L)_3=C_3^\pi\oplus L_2\xrightarrow{(\alpha_2\circ d_3^\pi,\phi)}
	Z_2(K)\to H_2(K).\]
	To see that this map is a cocycle~$\operatorname{Cone}(t_L)_3 \to H_2(K)$ note that since~$\alpha\circ t_L\simeq i_L+\phi\colon L_*|_{[0,2]}\to K_*$,  we have 
	\[\alpha_2\circ d_3^\pi\circ (t_L)_3
	=\alpha_2\circ (t_L)_2\circ d_3^L\simeq (i_L)_2\circ d_3^L+\phi\circ d_3^L
	=d_3^K\circ (i_L)_3 +\phi\circ d_3^L\equiv \phi\circ d_3^L\]
	in $H_2(K)$. 
	It follows that~$ \bsm \alpha_2 \circ d_3^\pi & \phi\esm \bsm d_4^\pi & (t_L)_3 \\ 0 & -d_3^L\esm=0.$
\end{construction}

We show in Lemma~\ref{lem:kWellDef} that the homology class of $\theta_{K,L}^t$ does not depend on the choice of~$(\alpha,\phi)$, thus leading to the following definition.	 

\begin{definition}
	\label{def:k-inv}
	Let~$K_*$ be an augmented chain complex that is acyclic in dimensions~$<2$,  let~$i_L\colon L_*\to K_*$ be a chain map,  let~$t \colon K_* \to C_*^\pi$ be a chain map 
	that preserves augmentations, set~$t_L:=t \circ i_L$,  and let~$\alpha\colon C_*^\pi|_{[0,2]}\to K_*|_{[0,2]}$ and $\phi\colon L_2\to Z_2(K)$ be as in \cref{lem:existence-alpha}. 
	Define the \emph{relative~$k$-invariant} of $(K,L)$ as
	$$k_{K,L}^t:=[\theta_{K,L}^t]\in H^3(\operatorname{Cone}(t_L);H_2(K)).$$
\end{definition}

The next lemma shows that $k_{K,L}^t$ is well defined.

\begin{lemma}
	\label{lem:kWellDef}
	The definition of~$k_{K,L}^t$ does not depend on the choice of~$(\alpha,\phi)$ as in \cref{lem:existence-alpha}.
\end{lemma}
\begin{proof}
The proof adapts~\cite[Proposition 7.1]{EilenbergMacLane-operators} from the absolute case to the relative case.
	Assume that~$\beta \colon C_*^\pi|_{[0,2]} \to K_*|_{[0,2]}$ and $\phi'\colon L_2\to Z_2(K)$ is another pair with the same properties as~$(\alpha,\phi)$.
	We will show that~$(\alpha_2 \circ d_3^\pi,\phi)$ and~$(\beta_2 \circ d_3^\pi,\phi')$ are homologous.
	Constructing a cochain~$\gamma \in C^2(\operatorname{Cone}(t_L);H_2(K))$ with $\delta^{\operatorname{Cone}(t_L)}(\gamma)=(\alpha_2 \circ d_3^\pi,\phi)-(\beta_2 \circ d_3^\pi,\phi')$ requires intermediate maps $E_2,D,D'$ and $G$ that we now introduce.

	Since~$\alpha$ and~$\beta$ preserve the augmentations and $K_*$ is acylic in dimensions $<2$,~\cite[Theorem~5.1]{EilenbergMacLane-operators} 
	ensures that there is a chain homotopy~$D \colon \alpha \simeq \beta$ in degrees~$0$ and~$1$.
	As explained in~\cite[Equation 3.6]{EilenbergMacLane-operators}, the map 
	$$E_2:=\beta_2-\alpha_2-D_1 \circ d_2^\pi\colon C_2^\pi \to K_2$$
	takes values in~$Z_2(K)$ and therefore defines a map~$E_2 \colon C_2^\pi \to H_2(K)$.
	Since~$d_2^\pi \circ d_3^\pi=0$,
	\[E_2 \circ d_3^\pi=\beta_2 \circ d_3^\pi-\alpha_2 \circ d_3^\pi\colon C_3^\pi\to H_2(K).\]
	Next, since~$\beta \circ t_L-\phi'$ and~$\alpha \circ t_L-\phi$ are chain homotopic, there exists maps~$D'_i\colon L_i\to K_{i+1}$ for~$i\leq 2$
	with~$(\beta_i-\alpha_i) \circ t_L=D'_{i-1} \circ d_i^L+d_{i+1}^K \circ D_i'$ for~$i\leq 1$, where~$D'_{-1}=0$,  and $(\beta_2-\alpha_2) \circ t_L+(\phi-\phi')=D'_{1} \circ d_2^L+d_{3}^K \circ D_2'$.
	In particular, note the equation~$d_1^K\circ D_0'=(\beta_0-\alpha_0) \circ (t_L)_0=d_1^K\circ D_0\circ (t_L)_0$. 
	Since~$K_*$ is acyclic in dimensions~$<2$,  we have~$\ker(d_1^K)=\im(d_2^K)$,  and there thus exists a map~$G\colon L_0\to K_2$ with~$d_2^K\circ G=D_0\circ (t_L)_0-D_0'$.
	
	Using the definition of~$E_2$ and noting that~$(\beta_2-\alpha_2) \circ (t_L)_2+(\phi-\phi') \equiv D_1' \circ d_2^L$ mod~$B_2(K)$, we deduce 
	\[ E_2 \circ (t_L)_2=(\beta_2-\alpha_2-D_1 \circ d_2^\pi)\circ (t_L)_2=(\phi'-\phi)+D_1' \circ d_2^L-D_1 \circ (t_L)_1 \circ d_2^L
	\colon C_2(L)\to H_2(K).\]
	Consider the map
	\[\gamma:=(E_2,-(D_1 \circ (t_L)_1-D_1'-G\circ d_1^L ))\colon C_2^\pi\oplus L_1\to K_2/B_2(K).\]
	Since $D_i \circ (t_L)_i$ and $D_i'$ both are chain homotopies between $\alpha\circ t_L$ and $\beta\circ t_L$ in degree $\leq 1$,  we have the equalities~$d_2^K\circ (D_1 \circ (t_L)_1-D_1'-G\circ d_1^L)=(D_0 \circ (t_L)_0-D_0') \circ d_1^L-d_2^K \circ G\circ d_1^L
	=0.$
	Since we already know that $\im(E_2) \subset Z_2(K)$, it follows that~$\gamma$ can be considered as a map to~$H_2(K)$.
	Finally,  we verify that~$\delta^{\operatorname{Cone}(t_L)}(\gamma)=(\beta_2 \circ d_3^\pi,\phi')-(\alpha_2 \circ d_3^\pi,\phi)$:
	\begin{align*}
		\delta^{\operatorname{Cone}(t_L)}(\gamma)
		&=(E_2,-(D_1 \circ (t_L)_1-D_1'-G \circ d_1^L)) \circ \begin{pmatrix}
			d_3^\pi & (t_L)_2 \\
			0 & -d_2^L
		\end{pmatrix} \\
		&=
		\left (\beta_2 \circ d_3^\pi-\alpha_2 \circ d_3^\pi,
		E_2 \circ (t_L)_2+(D_1 \circ (t_L)_1-D_1'))\circ d_2^L  \right) \\
		&=(\beta_2 \circ d_3^\pi,\phi')-(\alpha_2 \circ d_3^\pi,\phi).
	\end{align*}
	Thus~$(\alpha_2 \circ d_3^\pi,\phi)$ and~$(\beta_2 \circ d_3^\pi,\phi')$ represent the same cohomology class~$k_{K,L}^t\in H^3(\operatorname{Cone}(t_L);H_2(K))$.
\end{proof}

Next, we discuss the dependency on the choice of $t$ and on the free resolution $C_*^\pi$.
\begin{lemma}
\label{lem:Indepc}
Let $C_*^\pi,(C')^\pi $ be free resolutions of $\Z$, let~$K_*,K_*'$ be $\Z[\pi]$-chain complexes that are acyclic in dimensions~$<2$, let~$t \colon K_* \to C_*^\pi$ and~$t' \colon K_*' \to (C')^\pi $ be 
chain maps,  and let~$i_L\colon L_*\to~K_*$ and~$i_{L'}\colon L_*'\to K_*'$ be inclusions of subcomplexes.
Set~$t_L:=t \circ i_L$ and~$t_{L'}':=t'\circ i_{L'}$.

Given chain maps~$h\colon L_*\to L'_*$ and $g \colon C_*^\pi \to (C')^\pi_*$ such that~$g \circ t_L\simeq t_{L'}'\circ h$, there is, up to homotopy, a unique map~$ (g,h) \colon \operatorname{Cone}(t_L)_* \to \operatorname{Cone}(t_{L'}')_*$ that makes the following diagram commute:
$$
\xymatrix{
C_*^\pi \ar[d]^g\ar[r] & \operatorname{Cone}(t_L)_* \ar@{-->}[d]^-{(g,h)}\ar[r]& L_{*-1} \ar[d]^h\\
(C'_*)^\pi \ar[r]& \operatorname{Cone}(t'_{L'})_* \ar[r]& L_{*-1}' .
}
$$
Additionally,  $(g,h)$ satisfies the two following additional properties:
\begin{itemize}
\item if both $h$ and $g$ are chain homotopy equivalences, then so is $(g,h)$;
\item if $K_*=K'_*$ is an augmentable chain complex, $C_*^\pi=(C'_*)^\pi,g=\id_{C^\pi}$, and $t$ is augmentation-preserving,  then 
$$ (\id_{C_*^\pi},h)^*(k_{K,L'}^{t'})=k_{K,L}^t \in H^3(\operatorname{Cone}(t_L);H_2(K)).$$
\end{itemize}
\end{lemma}
\begin{proof}
The existence of the map $(g,h)$ follows from Lemma~\ref{lem:Cofibre}, as does the assertion on homotopy equivalences.
We therefore focus on the last statement: when $K=K'$ and $g=\id_{C^\pi}$,   a representative for~$(g,h)=(\id_{C^\pi},h)$ is $\bsm \id_{C^\pi} & \psi \\ 0& h \esm$ (where~$\psi \colon t_L\simeq t_{L'} \circ h$ is a homotopy) and the equality~$(\id_{C^\pi},h)^*(k_{K,L'}^{t'})=k_{K,L}^t$ follows either from a rapid verification of from the next proposition.
\end{proof}

Given subcomplexes $L_* \leq K_*$ and $L_*' \leq K_*'$, the next proposition shows that the relative $k$-invariant serves as an obstruction to extending a given chain map $h \colon L_* \to L_*'$ to~$K_*|_{[0,3]}\to K_*'|_{[0,3]}$.

\begin{proposition}
\label{prop:kConjChain}
Let~$K_*$ and~$K_*'$ be augmented $\Z[\pi]$-chain complexes which are acyclic in dimensions~$<2$, let~$t \colon K_* \to C_*^\pi$ and~$t'\colon K_*' \to C_*^\pi$ be augmentation-preserving chain maps,
let~$i_L\colon L_*\to K_*$ and~$i_{L'}\colon L_*'\to K_*'$ be inclusions of subcomplexes such that~$L_i$ (resp. $L_i'$) is a summand of~$K_i$ (resp. $K_i'$) for each~$i$, and set~$t_L:=t \circ i_L$ and~$t'_{L'}:=t'\circ i_{L'}$.

For a chain map~$h\colon L_*\to L'_*$ with~$t_L\simeq t'_{L'}\circ h$, and a homomorphism~$F\colon H_2(K)\to H_2(K')$,  the following assertions are equivalent:
\begin{itemize}
\item There exists an augmentation-preserving chain map~$f\colon K_*|_{[0,3]}\to K_*'|_{[0,3]}$ with
\[f\circ i_L|_{[0,2]}\simeq i_{L'}|_{[0,2]}\circ h|_{[0,2]}\colon L_*|_{[0,2]}\to K_*',\]
and~$f_*=F\colon H_2(K)\to H_2(K')$,
\item The relative $k$-invariants satisfy
$$h^*(k_{K',L'}^{t'})=F_*(k_{K,L}^t),$$
where $h^*:=(\id_{C^\pi},h)^*$ is the map from Lemma~\ref{lem:Indepc}.
\end{itemize}
\end{proposition}
\begin{proof}
We begin with the ``only if"-direction: we assume that $f$ exists and prove the equality involving relative $k$-invariants.
Let~$\alpha\colon C_*^\pi|_{[0,2]}\to K_*|_{[0,2]}$ and $\phi\colon L_2\to Z_2(K)$ be as in Lemma~\ref{lem:existence-alpha} (i.e. with $\alpha \circ t_L \simeq i_L+\phi$),  set~$\alpha':=f\circ \alpha\colon  C_*^\pi|_{[0,2]}\to K'_*|_{[0,2]}$ , choose a homotopy $\psi \colon t_L \simeq t'_{L'} \circ h$,  apply the last sentence of Lemma~\ref{lem:existence-alpha} to obtain a map~$\phi' \colon L_2\to Z_2(K)$ with $\alpha' \circ t'_{L'} \simeq i_{L'}+\phi'$ (this is possible because $f$ is augmentation-preserving) and consider the following diagram:
\begin{equation}
\label{eq:NecessaryCondition}
\xymatrix{
	C_3^\pi\oplus L_2\ar[r]^{(\alpha_2\circ d_3^\pi,\phi) }\ar[d]_{(\id_{C^\pi},h):=\bsm \id & \psi_2 \\ 0&h_2\esm}
	&Z_2(K)\ar[r]
	&H_2(K)\ar[d]^{f_*}\\
	C_3^\pi\oplus L'_2\ar[r]^{(\alpha_2'\circ d_3^\pi,\phi')}
	&Z_2(K')\ar[r]&H_2(K').
}
\end{equation}
A direct calculation shows that  this diagram commutes up to a coboundary:
\begin{align*}
\begin{pmatrix}
\alpha_2'\circ d_3^\pi&\phi'
\end{pmatrix}
\begin{pmatrix}
\id & \psi_2 \\ 0 & h_2
\end{pmatrix}
&=\begin{pmatrix}
\alpha_2' \circ d_3^\pi &  \alpha_2' \circ d_3^\pi \circ \psi_2+\phi'\circ h_2
\end{pmatrix} \\
&=\begin{pmatrix}
f \circ \alpha_2 \circ d_3^\pi &  \alpha_2' \circ t_L-\alpha_2'\circ t'_{L'} \circ h +\phi'\circ h_2
\end{pmatrix} 
+
\begin{pmatrix}
0 & \alpha_2'\circ \psi_1 \circ d_2^L
\end{pmatrix}  \\
&=\begin{pmatrix}
f \circ \alpha_2 \circ d_3^\pi & f \circ i_L+f\circ \phi -i_{L'} \circ h 
\end{pmatrix} 
+
\begin{pmatrix} 0& - \alpha_2'\circ \psi_1 \end{pmatrix} \begin{pmatrix} d_3^\pi &(t_L)_2 \\ 0&-d_2^L \end{pmatrix}
 \\
&=\begin{pmatrix}
f \circ \alpha_2 \circ d_3^\pi &f\circ \phi
\end{pmatrix}
-
\begin{pmatrix} 0&  \alpha_2'\circ \psi_1 \end{pmatrix} \begin{pmatrix} d_3^\pi &(t_L)_2 \\ 0&-d_2^L \end{pmatrix}.
\end{align*}
The third equality uses $\alpha \circ t_L \simeq i_{L}+\phi$ and $\alpha' \circ t'_{L'} \simeq i_{L'}+\phi'$,  the fourth uses $f\circ i_L|_{[0,2]}\simeq i_{L'}|_{[0,2]}\circ h|_{[0,2]}$.
In both cases, we are also relying on the fact that the equalities take place in~$H_2(K').$

Since the top composition in~\eqref{eq:NecessaryCondition} is~$\theta_{K,L}^t$, which represents $k_{K,L}^t$ and the bottom composition is~$\theta_{K',L'}^{t'}$, which represents $k_{K',L'}^{t'}$,  it follows that $h^*(k_{K',L'}^{t'})=F_*(k_{K,L}^t)$ as claimed.

We now show the ``if"-direction. 
The strategy of the proof, which is a relative version of~\cite[Theorem 7.1]{EilenbergMacLane-operators} and~\cite[Theorem 4]{MacLaneWhitehead} is as follows.
We define intermediate maps~$g$ and~$F_2$ that will allow us to construct~$f_i$ in degree~$0,1,2$.
We then argue that~$f_2$ extends over the~$3$-skeleton and then conclude by showing that the resulting chain map~$f$ satisfies the required properties.

We construct the map~$g.$
Let~$\theta_{K,L}^t$ and~$\theta_{K',L'}^{t'}$ respectively be the representatives of~$k_{K,L}^t$ and~$k_{K',L'}^{t'}$ from \cref{def:k-inv}.
Since we assumed~$F_*(k_{K,L}^t)-h^*(k_{K',L'}^{t'})=0 \in H^3(\operatorname{Cone}(t_L);H_2(K'))$, on the chain level,  the representative cocycle~$F \circ \theta_{K,L}^t-\theta_{K',L'}^{t'} \circ \bsm \id & \psi \\ 0 & h \esm$ (where $\psi \colon t_L \simeq t_L' \circ h$ is a homotopy) is in fact a coboundary, meaning that there exists a map 
\[m' \colon  \operatorname{Cone}(t_L)_2 \to H_2(K') \quad \text{with} \quad
m'\circ d^{\operatorname{Cone}(t_L)}_3=F \circ \theta_{K,L}^t-\theta_{K',L'}^{t'} \circ \bsm \id & \psi \\ 0 & h \esm \colon \operatorname{Cone}(t_L)_3\to H_2(K').\]
Since~$\operatorname{Cone}(t_L)_2$ is a free~$\Z[\pi]$-module,~$m'$ lifts to a map~$g' \colon  \operatorname{Cone}(t_L)_2 \to Z_2(K')$ and we set
$$g_\pi'\colon C^\pi_2\xrightarrow{\incl_L} \operatorname{Cone}(t_L)_2 \xrightarrow{g'}Z_2(K').$$
We construct the map~$F_2$.
Consider the
chain maps~$\alpha\colon C_*^\pi|_{[0,2]}\to K_*|_{[0,2]}$ and~$\alpha'\colon C_*^\pi|_{[0,2]}\to~K_*'|_{[0,2]}$ with~$\alpha\circ t_L\simeq i_{L}+\phi\colon L_*|_{[0,2]}\to K_*$ (say via a chain homotopy~$D_L$) and~$\alpha'\circ t'_{L'}\simeq i_{L'}+\phi'\colon L'|_{[0,2]}\to K'$
underlying the definition of $\theta_{K,L}^t$ and $\theta_{K',L'}^{t'}$.
By \cref{lem:EilenbergMaclane51},~$\alpha\circ t\simeq \id_{K}\colon K_*|_{[0,1]}\to K_*$ and there exists a chain homotopy~$D$ such that~$D\circ i_L$ is part of the chain homotopy~$D_L\colon \alpha\circ t_L\simeq i_{L}+\phi\colon L_*|_{[0,2]}\to K_*$.
Consider the map
$$ E_2 \colon K_2 \to  Z_2(K),  \quad E_2:=\id-\alpha_2 \circ t_2-D_1 \circ d^K_2.$$
Note that the composition~$L_2\xrightarrow{(i_L)_2}K_2\xrightarrow{E_2}Z_2(K)\to H_2(K)$ equals $-\phi$:
\[\phi+E_2\circ (i_L)_2
=\phi+(i_L)_2-\alpha_2 \circ (t_L)_2-D_1 \circ d^K_2 \circ (i_L)_2
=(i_L)_2+\phi-\alpha_2 \circ (t_L)_2-(D_L)_1 \circ d_2^L
=d_3^K \circ (D_L)_2.\]
Since~$Z_2(K') \to H_2(K')$ is surjective and~$K_2$ is a free~$\Z[\pi]$-module,  there is a map
$$F_2 \colon K_2 \to Z_2(K')$$
making the following diagram commute:
$$
\xymatrix{
	K_2\ar[r]^-{E_2}\ar@{-->}[rd]^{F_2}&Z_2(K)\ar[r]^{\proj}&H_2(K)\ar[d]^F \\
	&Z_2(K')\ar[r]^{\proj'}&H_2(K').
}
$$
Since~$L_2\xrightarrow{(i_L)_2}K_2\xrightarrow{E_2}Z_2(K)\to H_2(K)$ is $-\phi$, we can assume that~$F_2$ is $-F \circ \phi$ on the summand~$L_2$ of~$K_2$, i.e.~$\proj'\circ F_2\circ (i_L)_2=-F \circ \phi$.
For later use, it will be helpful to rewrite this expression.
Recall that $m' \colon  \operatorname{Cone}(t_L)_2 \to H_2(K')$ lifts to $g'=(g'_\pi,g_L') \colon  \operatorname{Cone}(t_L)_2 \to K_2'$.
Expand the equation~$m'\circ d^{\operatorname{Cone}(t_L)}_3=F \circ \theta_{K,L}^t-\theta_{K',L'}^{t'} \circ \bsm \id & \psi \\ 0 & h \esm \colon \operatorname{Cone}(t_L)_3\to H_2(K')$,  and look at the second coordinate of the outcome:
$$ \proj'\circ F_2\circ (i_L)_2
=-F \circ \phi
=-\proj'\circ (\phi' \circ h_2-\alpha_2' \circ d_3^\pi \circ \psi_2+g'_\pi \circ (t_L)_2 -g_L' \circ d^L_2) \colon L_2 \to H_2(K').$$
Now lift to $Z_2(K')$ to obtain, for some homomorphism $\varpi \colon L_2 \to K_3'$, the expression
$$ F_2\circ (i_L)_2
=-\phi' \circ h_2+\alpha_2' \circ d_3^\pi \circ \psi_2-g'_\pi \circ (t_L)_2 +g_L' \circ d^L_2 + d_3^K \circ \varpi \colon L_2 \to Z_2(K').$$
Now that we have defined~$g_\pi'$ and~$F_2$, we are ready to define~$f_0,f_1$ and~$f_2$.	
We proceed by slight modifying the argument from~\cite[Proof of Theorem~7.1]{EilenbergMacLane-operators} (see also \cite[Proof of Theorem~4]{MacLaneWhitehead}): on the~$2$-skeleton of our chain complexes,  the required chain map is defined by
\begin{align*}
	&f_i := \alpha_i' \circ t_i\colon K_i \to K_i' \quad \text{ for } i=0,1,\\
	&f_2 := \alpha_2' \circ t_2  +g_\pi' \circ t_2 +F_2 \colon K_2 \to K_2'.
\end{align*}
We verify that~$f\circ \iota_L \simeq \iota_{L'} \circ h$ as maps~$L_*|_{[0,2]}\to K_*$,  then check~$f_2$ induces~$F$ on~$H_2(K)$, and finally confirm that~$f_2$ extends to the~$3$-skeleta of the chain complexes.		

We begin by showing that~$f\circ \iota_L \simeq \iota_{L'} \circ h$.
Starting from the definition of $f$ and our calculation of $F_2\circ (i_L)_2$,  we consider the following sequence of chain homotopies of maps~$L_*|_{[0,2]}\to K_*$:
\begin{align*}
f\circ \iota_L
&=\alpha' \circ t_L +g_\pi' \circ (t_L)_2 +F_2 \circ (i_L)_2 \\
&= \alpha' \circ t_L +g_\pi' \circ (t_L)_2-(\phi' \circ h_2-\alpha_2' \circ d_3^\pi \circ \psi_2+g'_\pi \circ (t_L)_2-g_L' \circ d^L_2 -d_3^K \circ \varpi) \\
&= \alpha' \circ t_L +\alpha_2' \circ d_3^\pi \circ \psi_2-\phi' \circ h_2+g_L' \circ d^L_2 +d_3^K \circ \varpi \\
&\simeq \alpha' \circ t_L +\alpha_2' \circ d_3^\pi \circ \psi_2-\phi' \circ h_2+g_L' \circ d^L_2  \\
&\simeq \alpha' \circ t_L +\alpha_2' \circ d_3^\pi \circ \psi_2-\phi' \circ h_2 \\
&\simeq  \alpha' \circ t'_{L'} \circ h -\phi' \circ h_2 \\
&\simeq i_{L'} \circ h.
\end{align*}
The first chain homotopy, which removes $d_3^K \circ \varpi$,  is~$-\varpi \colon L_2 \to K_3'$ in degree $2$ and the zero map in all other degrees.
The second chain homotopy, which removes~$g_L' \circ d^L_2$,  is~$g_L' \colon L_1 \to K_2'$ in degree~$1$ and the zero map in all other degrees; this uses that $g_L'$ takes values in $Z_2(K')$ so that the map on~$L_1$ remains unchanged under the homotopy.
The third chain homotopy is~$-\alpha_i'\circ\psi_i$ in degrees~$i=0,1$ and the zero map in degree $2$; here recall that~$\psi \colon t_L \simeq t_L' \circ h$.
The final chain homotopy is a consequence of the chain homotopy~$\alpha' \circ t_{L'}' \simeq i_{L'}+\phi'$. 

\medskip

We verify that the chain map~$f_2$ induces~$F$ on~$H_2(K)$ i.e.\  that for every cycle~$z \in Z_2(K)$ we have~$[f_2(z)]=F([z])$.
Note that since~$[t_2(z)] \in H_2(C^\pi)=0$, there exists a~$c \in C_3^\pi$ with~$d^\pi_3 (c)= t_2(z)$.
Using the definitions of~$f_2,F_2,E_2$ and~$g_\pi'$  we obtain:
\begin{align*}
	[f_2(z)]
	&=[(\alpha_2' \circ t_2 +g_\pi' \circ t_2+F_2)(z)] \\
	&=[\alpha_2' \circ d^\pi_3(c)
	+g_\pi' \circ d^\pi_3(c)]
	+F([E_2(z)]) \\
	&=[\theta_{K'}(c)
	+m' \circ \incl_L \circ d^\pi_3(c)]
	+F \circ \proj \circ (\id-\alpha_2 \circ t_2-D_1 \circ d^K_2)(z) \\
	&=[\theta_{K'}(c)
	+m' \circ d^{\operatorname{Cone}(t_L)}_3 \circ \incl_L(c)]
	+F([z])-F([\alpha_2 \circ d_3^\pi(c)]) \\
	&=[\theta_{K'}(c)
	+(F \circ \theta_{K,L}^t-\theta_{K',L'}^{t'} \circ \bsm \id & \psi \\ 0 & h \esm)(\incl_L(c))]
	+F([z])-[F(\theta_K(c))] \\
	&=[\theta_{K'}(c)
	+(F \circ \theta_K-\theta_{K'})(c)]
	+F([z])-[F(\theta_K(c))] \\
	&=F([z]).
\end{align*}
We now prove that~$f_2$ extends to a map~$K_3 \to K_3'$.
First, a near identical calculation to the above shows that~$[f_2 \circ d^K_3(c)]=0$ for all~$c\in K_3$. 
To see this, write~$\proj\colon Z_2(K)\to H_2(K)$ and~$\proj'\colon Z_2(K')\to H_2(K')$ for the projections and calculate
\begin{align*}
	\proj' \circ f_2 \circ d^K_3
	&=\proj' \circ ( \alpha_2' \circ t_2 +g'_\pi  \circ  \incl_L \circ t_2+F_2)\circ d^K_3 \\
	&=\proj'\circ \alpha_2'  \circ  d^\pi_3  \circ  t_3+
	m'  \circ  \incl_L \circ t_2 \circ d^K_3+
	F \circ \proj  \circ  E_2 \circ d^K_3 \\
	&=[\theta_{K'} \circ t_3+
	m'  \circ \incl_L  \circ  d^\pi_3  \circ  t_3]+
	F \circ \proj  \circ  (\id-\alpha_2 \circ  t_2-D_1 \circ d_2^K ) \circ d^K_3 \\
	&=[\theta_{K'} \circ t_3+
	m'  \circ  d^{\operatorname{Cone}(t_L)}_3 \circ \incl_L \circ t_3]-
	F  \circ  \proj  \circ  \alpha_2  \circ  t_2  \circ  d^K_3 \\
	&=[\theta_{K'} \circ t_3+
	\left(F \circ \theta_{K,L}-\theta_{K',L'} \circ \bsm \id & \psi \\ 0 & h \esm\right) \circ  \incl_L  \circ  t_3]-
	F([\theta_K \circ  t_3]) \\
	&=[\theta_{K'} \circ t_3+
	((F \circ \theta_K-\theta_{K'}) \circ t_3)]-
	F([\theta_K \circ t_3]) \\
	&=0.
\end{align*}
It follows that~$f_2 \circ d^K_3(K_3) \subset  B_2(K')$.
Since~$K_3$ is free and~$d^{K'}_3 \colon K_3' \to  B_2(K')$ is surjective,  there is a map~$f_3 \colon K_3 \to K_3'$ with~$d_3^{K'}  \circ f_3=f_2  \circ d^K_3$.
\end{proof}

We conclude this section with a technical result that we will require later on.
To do so, recall that given a homomorphism $u \colon \pi \to \pi'$ and a~$\Z[\pi']$-module~$H$, we write~$\Res_u H$ for the~$\Z[\pi]$-module whose underlying abelian group is~$H$, but with~$\pi$-action given by~$\gamma \cdot x=u(\gamma)x$.

\begin{lemma}
\label{lem:IgnoreRes}
Let~$u \colon \pi \to \pi'$ be a group isomorphism,  let~$K_*'$ be an augmented~$\Z[\pi']$-chain complex that is acyclic in dimensions~$<2$,  let~$i_{L'} \colon L_*'\to K_*'$ be the inclusion of a subcomplex,
let~$C^{\pi'}$ be a free~$\Z[\pi']$-resolution of~$\Z$,  let~$t'\colon K_*' \to C_*^{\pi'}$ be a~$\Z[\pi']$-chain map that preserves augmentations and set $t'_{L'}:=t' \circ i_{L'}.$

The~$\Z[\pi]$-chain complex~$\Res_u(K_*')$ is acyclic in dimensions~$<2$ and,  considering the~$\Z[\pi]$-chain map~$\Res_u(t')\colon \Res_u(K_*') \to \Res_u(C_*^{\pi'})$, there is a canonical group isomorphism
$$  H^3(\operatorname{Cone}(\Res_u(t'_{L'})); H_2(\Res_u K'))
 \xrightarrow{\cong}  H^3(\operatorname{Cone}(t'_{L'}); H_2(K'))~$$
 that takes~$k_{(\Res_u(K'),\Res_u(L'))}^{\Res_u(t')}$ to~$k_{(K',L')}^{t'}$.
\end{lemma}
\begin{proof}
Restriction of scalars does not change the chain groups and differentials of a chain complex.
Since~$\operatorname{Cone}(\Res_u(t'_{L'}))=\Res_u\operatorname{Cone}(t'_{L'})$,  one can identify~$\Hom_{\Z[\pi]}(\operatorname{Cone}(\Res_u(t'_{L'})),\Res_u H_2(K'))$ with~$\Hom_{\Z[\pi]}(\operatorname{Cone}(t'_{L'}),H_2(K'))$.
The claim about the~$k$-invariants also follows immediatly.
\end{proof}

\begin{remark}
Since restriction of scalars is manifestly an exact functor, we will often identify the~$\Z[\pi]$-modules~$\Res_u H_2(K')$ and~$H_2(\Res_u K')$, as well as~$H^3(\operatorname{Cone}(\Res_u(t')); H_2(\Res_u K'))$ and~$H^3(\operatorname{Cone}(\Res_u(t')); \Res_u H_2(K'))$, as abelian groups.
\end{remark}

\section{Relative $k$-invariants of CW pairs.}
\label{sub:kinvariantCW}

This section defines the $k$-invariant of a CW pair and describes some of its main properties.

\begin{notation}
In what follows, for a CW pair~$(X,Y)$ with basepoint $x \in Y$,  universal cover $\widetilde{X}$ and restricted cover~$\widetilde{Y}$, we fix a preimage $\wt x \in \wt Y$ of $x$.
The complex~$C_*(\widetilde{X})$ comes with a natural augmentation obtained by basing~$C_0(\widetilde{X})$ with any lifts~$\widetilde{v}_0,\ldots, \widetilde{v}_n$ of the~$0$-cells~$v_0,\ldots,v_n$ of~$X$ and setting~$\aug(\gamma\widetilde{v}_i)=1$ for every~$\gamma \in \pi.$
Using the basepoint~$\wt x$, the Hurewicz theorem leads to an isomorphism~$\pi_2(X) \cong  H_2(\widetilde{X})\cong H_2(C_*(\widetilde{X}))$.
	\end{notation}
	
	Section~\ref{sub:SetUpRelativeKCW} focuses on some lemmas that are required to establish properties of the $k$-invariant of a CW pair.
	Section~\ref{sub:kCWDef} defines the $k$-invariant of a CW pair.
	
	\subsection{Set-up and lemmas}
	\label{sub:SetUpRelativeKCW}

We prove several lemmas that will be of use to define the $k$-invariant of a CW pair.
First, we note that we are in the setting of Section~\ref{sub:PrepKInvariant}.

\begin{lemma}
Given a CW complex $X$ and a free $\Z[\pi_1(X)]$-resolution~$C^{\pi_1(X)}_*$ of~$\Z$, there is a chain map $t \colon C_*(\widetilde{X}) \to C^{\pi_1(X)}_*$ that preserves augmentations.
\end{lemma}
\begin{proof}
Pick a model~$B\pi_1(X)$ with the same~$0$-skeleton as~$X$ and containing $X$ as a subcomplex, choose a chain homotopy equivalence~$C_*(\widetilde{B}\pi_1(X)) \simeq C^{\pi_1(X)}_*$ that preserves augmentations, and consider the composition~$t \colon C_*(\widetilde{X}) \to  C_*(\widetilde{B}\pi_1(X)) \simeq C^{\pi_1(X)}_*.$
\end{proof}

The next two lemmas will be used to study the relative $k$-invariant and its properties.

\begin{lemma}
\label{lem:InducedMapkInvariants}
Let $\pi$ and $\pi'$ be groups, let~$Y, Y'$ be CW complexes, and let~$\nu \colon Y\to B\pi$ as well as~$\nu'\colon Y'\to B\pi'$ be cellular maps.
Given a map~$h\colon Y \to Y'$, an isomorphism $u \colon \pi \to \pi'$ such that~$u\circ \nu_*=\nu'_*\circ h_*$,  and a map~$v \colon B\pi \to B\pi'$ realising $u$, there is, up to homotopy, a unique chain map
$$(u,h)^v \colon C_*(\widetilde{M}(\nu),\widetilde{Y}) \to \Res_u C_*(\widetilde{M}(\nu'),\widetilde{Y}')$$
 that makes the following diagram commute:
\[
\xymatrix{
C_*(\widetilde{B}\pi) \ar[d]^{v_*}\ar[r] & C_*(\widetilde{M}(\nu),\widetilde{Y})  \ar@{-->}[d]^-{(u,h)^v}\ar[r]& C_{*-1}(\widetilde{Y}) \ar[d]^{h_*}
\\
\Res_u C_*(\widetilde{B}\pi') \ar[r]& \Res_u C_*(\widetilde{M}(\nu'),\widetilde{Y}')  \ar[r]& \Res_u C_{*-1}(\widetilde{Y}').
}
\]
A homotopy equivalence $v$ as above exists, is unique up to homotopy,  and if~$v \simeq v'$ are homotopic, then $(u,h)^v \simeq (u,h)^{v'}$.
In particular, for any $\Z[\pi']$-module $A$,  the pair $(u,h)$ induces a group homomorphism 
	$$(u,h)^*\colon H^3(M(\nu'),Y';A)\to H^3(M(\nu),Y;\Res_u A).$$
\end{lemma}
\begin{proof}
Since~$\nu$ and~$\nu'$ are cellular,  we have~$C_*(\widetilde{M}(\nu),\widetilde{Y})=\operatorname{Cone}(\widetilde{\nu}_*)_*$ and~$C_*(\widetilde{M}(\nu'),\widetilde{Y})=\operatorname{Cone}(\widetilde{\nu}'_*)_*$; see e.g.~\cite[page 20]{Weibel}.
Since any two choices $v,v' \colon B\pi \to B\pi'$ realising $u$ are homotopic, the result follows directly from Lemma~\ref{lem:Cofibre}.
In this last step, we used Lemma~\ref{lem:IgnoreRes} to identify $H^*(\Res_u(C(M(\nu),Y));\Res_uA)$ with~$H^*(M(\nu),Y;A).$
\end{proof}

\begin{lemma}
		\label{lem:Identifying-targets}
Let $\pi$ be a group, let $Y$ be a CW complex and let~$C^{\pi}_*$ be a free $\Z[\pi]$-resolution of~$\Z$.
Given a cellular map~$\nu \colon Y\to B\pi$,  a chain map~$t \colon C_*(\widetilde{Y}) \to C^{\pi}_*$ that preserves augmentations,  and a chain homotopy equivalence $\varphi \colon C_*(\widetilde{B}\pi) \to C_*^\pi$ with $\varphi \circ \nu_* \simeq t_Y$,
there is up to homotopy a unique map 
$$ \phi_{t,\nu}^\varphi \colon C_*(\widetilde{M}(\nu),\widetilde{Y}) \to \operatorname{Cone}(t_Y)_*$$
 that makes the following diagram commute:
\[
\xymatrix{
C_*(\widetilde{B}\pi) \ar[d]^\varphi\ar[r] & C_*(\widetilde{M}(\nu),\widetilde{Y})  \ar@{-->}[d]^-{\phi_{t,\nu}^{\varphi}}\ar[r]& C_{*-1}(\widetilde{Y})  \ar[d]^=\\
C_*^\pi \ar[r]& \operatorname{Cone}(t)_*  \ar[r]& C_{*-1}(\widetilde{Y}) .
}
\]
A homotopy equivalence $\varphi$ as above exists, is unique up to homotopy, and if~$\varphi \simeq \varphi'$ are homotopic, then $\phi_{t,\nu}^{\varphi} \simeq \phi_{t,\nu}^{\varphi'}$.
In particular, for every~$\Z[\pi]$-module $A$,  there is an isomorphism 
	$$\phi_{t,\nu}^*\colon H^3(\operatorname{Cone}(t);A)\xrightarrow{\cong} H^3(M(\nu),Y;A).$$
	Furthermore,  this map satisfies the following properties:
	\begin{enumerate}
	\item for any chain homotopy $t\simeq t'\colon C_*(\widetilde{Y}) \to C^{\pi}_*$, we have 
	$$\phi_{t,\nu}^*\circ (\id_{C^\pi},\id_{C(\widetilde{Y})})^*=\phi_{t',\nu}^*$$ where $(\id_{C^\pi},\id_{C(\widetilde{Y})})^*$ is the isomorphism from \cref{lem:Indepc};
	\item for every cellular map $\nu'\colon Y\to B\pi$ that is homotopic to $\nu$,  we have 
	$$\phi_{t,\nu}^*= (\id_\pi,\id_Y)^*\circ\phi_{t,\nu'}^*$$ 	
	where $(\id_\pi,\id_Y)^*$ is the isomorphism from \cref{lem:InducedMapkInvariants}.
		\end{enumerate}
\end{lemma}
\begin{proof}
For the existence of $\varphi$, note that since $C_*^\pi$ is acyclic,~\cite[Theorem 5.1]{EilenbergMacLane-operators} ensures that any augmentation preserving chain map~$\varphi \colon C_*(\widetilde{B}\pi) \to C_*^\pi$ satisfies $\varphi \circ \nu_* \simeq t_Y$.
The uniqueness of~$\varphi$ up to homotopy equivalence follows from Lemma~\ref{lem:EilenbergMaclane51}.
Since~$\nu$ is cellular,  we use the identification~$C_*(\widetilde{M}(\nu),\widetilde{Y})=\operatorname{Cone}(\widetilde{\nu}_*)_*$.
The existence and uniqueness of~$\phi_{t,\nu}^\varphi$ follow from Lemma~\ref{lem:Indepc} applied to $K_*=C_*(\widetilde{B}\pi)=K'_*,L_*=C_*(\widetilde{Y})=L'_*$ and $h=\id_{C(\widetilde{Y})}$.
In order to prove the last two assertions, pick $\varphi,\varphi' \colon C_*(\widetilde{B}\pi) \to C_*^\pi$ as above.
Since these maps are necessarily homotopic, we obtain
the homotopy commutative diagram
\[ 
\xymatrix{
C_*(\widetilde{B}\pi) \ar[r]^=\ar[d]^{\varphi}&C_*(\widetilde{B}\pi) \ar[d]^{\varphi'}\\
C_*^\pi \ar[r]^{=}&C_*^\pi.
}
\]
The uniqueness portion of Lemma~\ref{lem:Cofibre} now yields the homotopy commutative diagram 
\[ 
\xymatrix{
C_*(\widetilde{M}(\nu),\widetilde{Y}) \ar[r]^{(\id_\pi,\id_Y)}\ar[d]^{\phi_{t,\nu}^\varphi}& C_*(\widetilde{M}(\nu'),\widetilde{Y})  \ar[d]^{\phi_{t',\nu'}^{\varphi'}}\\
\operatorname{Cone}_*(t|_Y) \ar[r]^{(\id_{C^\pi},\id_{C(\widetilde{Y})})} &\operatorname{Cone}(t'|_Y)_*.
}
\]
Here the notation for the horizontal maps is the one from Lemmas~\ref{lem:Indepc} and~\ref{lem:InducedMapkInvariants}.
When $t=t'$, the bottom map can be taken to be the identity, whereas when $\nu=\nu'$, the top map can be taken to be the identity.
The last two assertions now follow by passing to cohomology.
\end{proof}

\subsection{The $k$-invariant of a CW pair}
	\label{sub:kCWDef}

We define the $k$-invariant of a CW pair and establish some of its properties.

\begin{notation}
Given a CW pair~$(X,Y)$,  a cellular map~$\nu \colon X \to B\pi_1(X)$ 
and an augmentation-preserving chain map~$t \colon C_*(\widetilde{X}) \to C^{\pi_1(X)}_*$,  we write
$$ k_{X,Y}^\nu:=\phi_{t,\nu}^*(k_{C(\widetilde{X}),C(\widetilde{Y})}^t) \in H^3(M(\nu|_Y),Y;\pi_2(X)).$$
Here and in what follows,  for brevity, we write $\phi_{t,\nu}^*:=\phi_{t|_{C(\widetilde{Y})},\nu|_Y}^*$ for the map from Lemma~\ref{lem:Identifying-targets}.
\end{notation}

Lemmas~\ref{lem:Identifying-targets} and~\ref{lem:Indepc}
imply that $k_{X,Y}^\nu$ does not depend on the choice of the chain map~$t$.
Note also that for any two cellular maps~$\nu,\nu' \colon X \to B\pi_1(X)$,
 Lemma~\ref{lem:Identifying-targets} shows that the~$k$-invariants~$k_{X,Y}^\nu$ and~$k_{X,Y}^{\nu'}$ are related by a canonical isomorphism.

\begin{definition}
\label{def:RelativekinvariantCW}
The \emph{relative $k$-invariant} of a CW pair~$(X,Y)$ refers to 
$$k_{X,Y}^\nu \in  H^3(M(\nu|_Y),Y;\pi_2(X))$$
 for any choice a cellular~$\nu \colon X \to B\pi_1(X)$.
\end{definition}

In our conventions, we fixed for each space $X$ a model $B\pi_1(X)$ such that $X \to B\pi_1(X)$ is the identity on $\pi_1$.
We nevertheless note that the proof of Lemma~\ref{lem:InducedMapkInvariants} shows that if one uses another such model, say $(B\pi_1(X))'$, the resulting $k$-invariants are related by a canonical isomorphism.

\begin{remark}
The reader might have noted that the choice of a map~$\nu \colon X \to B\pi_1(X)$ induces an augmentation preserving chain map~$t=\nu_* \colon C_*(\widetilde{X}) \to C_*(\widetilde{B}\pi_1(X))$.
The reason for which we do not simply take~$C_*^\pi=C_*(\widetilde{B}\pi_1(X))$ and~$t=\nu_*$ will become apparent during the proof of the next proposition when we compare the relative~$k$-invariants of two different CW pairs.
\end{remark}

In what follows, $i \colon Y \to X$ denotes the inclusion of a subcomplex.
The next result constitutes one of the main technical steps needed to prove Theorem~\ref{thm:RealiseAlgebraic3Type-intro}.

\begin{theorem}
\label{thm:GoalRelativekCW}
Let~$(X,Y),(X',Y')$ be CW pairs,  let $\nu \colon X \to B\pi_1(X)$ and~$\nu' \colon X \to B\pi_1(X')$ be cellular maps,
let $u\colon \pi_1(X)\to \pi_1(X')$ be an isomorphism, let~$h \colon Y \to Y'$ be a map that satisfies~$u\circ i_* = i_*' \circ h_*$, and let~$F \colon \pi_2(X) \to \pi_2(X')$ be a $u$-equivariant homomorphism.
The following assertions are equivalent:
\begin{itemize}
\item There is a map~$g \colon P_2(X) \to P_2(X')$ that induces $u$ and $F$ on $\pi_1$ and~$\pi_2$, respectively, and such that the following are homotopic:
 \[Y\hookrightarrow X\to P_2(X)\xrightarrow{g}P_2(X')\qquad\text{and}\qquad Y\xrightarrow{h}Y'\hookrightarrow X'\to P_2(X').\] 
\item The relative $k$-invariants satisfy
$$(u,h)^*(k_{X',Y'}^{\nu'})=F_*(k_{X,Y}^\nu) \in H^3(M(\nu|_Y),Y;\Res_u\pi_2(X')),$$
for any $\nu \colon X \to B\pi_1(X)$ and $\nu' \colon X' \to B\pi_1(X')$, and where $(u,h)^*$ is the map from \cref{lem:InducedMapkInvariants}.
\end{itemize}
\end{theorem}
\begin{proof}
We begin with some set-up.
Set~$\pi:=\pi_1(X)$ and~$\pi':=\pi_1(X')$ for brevity.
The assumption~$u\circ i_* = i_*' \circ h_*$ makes it possible to choose a homotopy equivalence~$v \colon B\pi \to B\pi'$ that induces~$u$ on~$\pi_1$, satisfies~$v\circ \nu|_Y \simeq \nu'|_Y \circ h$, and induces an augmentation-preserving $\Z[\pi]$-chain homotopy equivalence on the chain level.
Next, choose an augmentation preserving homotopy equivalence~$\varphi' \colon \Res_u C_*(\widetilde{B}\pi') \simeq C_*^\pi$,  leading to an augmentation preserving homotopy equivalence~$\varphi:=\varphi' \circ v_* \colon C_*(\widetilde{B}\pi) \simeq C_*^\pi.$
This set-up can be summarised by the following diagram of~$\Z[\pi]$-chain complexes in which the triangle commutes and the rectangle homotopy commutes:
\[
\xymatrix{
C_*(\widetilde{Y})\ar[rr]^-{h_*}\ar[d]&&\Res_u C_*(\widetilde{Y}')\ar[d]  \\
C_*(\widetilde{X})\ar[d]&&\Res_u C_*(\widetilde{X}') \ar[d] \\
C_*(\widetilde{B}\pi)\ar[rr]^-{v_*,\simeq}\ar[rd]_{\varphi,\simeq}&&\Res_u C_*(\widetilde{B}\pi')\ar[ld]^{\varphi',\simeq} \\
&C_*^\pi.&
 }
\]
We obtain $\Z[\pi]$-chain maps~$t \colon C_*(\widetilde{X}) \to C_*^\pi$ and~$t_u' \colon \Res_u C_*(\widetilde{X}') \to C_*^\pi$ with~$t_Y \simeq (t_u')_{Y'} \circ~h_*$, where we write $t_Y:=t|_{C_*(\widetilde{Y})}$ for brevity, and similarly for $(t_u')_{Y'}$.
Proposition~\ref{prop:kConjChain} ensures that~$h_* \colon C_*(\widetilde{Y})|_{[0,2]} \to \Res_u C_*(\widetilde{Y}')|_{[0,2]}$ extends to a chain map~$C_*(\widetilde{X})|_{[0,2]} \to \Res_u C_*(\widetilde{X}')|_{[0,2]}$ that induces $F$ if and only if~$F_*(k_{C(\widetilde{X}),C(\widetilde{Y})}^t)=(\id_{C^\pi},h)^*(k_{\Res_u C(\widetilde{X}'),\Res_u C(\widetilde{Y}')}^{t_u'})$.
\begin{claim}
\label{claim:key}
The following equivalence holds:
$$F_*(k_{C(\widetilde{X}),C(\widetilde{Y})}^t)=(\id_{C^\pi},h)^*(k_{\Res_u C(\widetilde{X}'),\Res_u C(\widetilde{Y}')}^{t_u'}) \quad \text{if and only if} \quad F_*(k_{X,Y}^\nu)=(u,h)^*(k_{X',Y'}^{\nu'}).$$
\end{claim}
\begin{proof}
Consider the following commutative diagram:
\[
\xymatrix{
C_*^\pi\ar[r]^=
&C_*^\pi \ar[r]^=
&C_*^\pi \\
C_*(\widetilde{B}\pi) \ar[r]^-{=}\ar[u]^{\varphi}
&C_*(\widetilde{B}\pi) \ar[r]^-{v_*}\ar[u]^{\varphi}
&\Res_u C_*(\widetilde{B}\pi). \ar[u]^{\varphi'}
}
\]
Applying Lemma~\ref{lem:Cofibre}, and recalling that $C_*(\widetilde{M}(\nu|_Y),\widetilde{Y})=\operatorname{Cone}((\widetilde{\nu}|_Y)_*)_*,C_*(\widetilde{M}(\nu'|_{Y'}),\widetilde{Y}')=\operatorname{Cone}((\widetilde{\nu}'|_{Y'})_*)_*$ as well as the notation from Lemmas~\ref{lem:Indepc},~\ref{lem:InducedMapkInvariants} and~\ref{lem:Identifying-targets} leads to the following homotopy commutative diagram:
\[
\xymatrix{
\operatorname{Cone}(t_Y)_*\ar[r]^=
&\operatorname{Cone}(t_Y)_* \ar[r]^-{(\id,h_*)}
&\operatorname{Cone}((t_u')_{Y'})_* \\
C_*(\widetilde{M}(\nu|_Y),\widetilde{Y}) \ar[r]^-{=}\ar[u]^-{\phi_{t,\nu}^\varphi}
&C_*(\widetilde{M}(\nu|_Y),\widetilde{Y})  \ar[r]^-{(v_*,h_*)} \ar[u]^-{\phi_{t,\nu}^\varphi}
&\Res_u C_*(\widetilde{M}(\nu'|_{Y'}),\widetilde{Y}').  \ar[u]^-{\phi_{t_u',\nu'}^{\varphi'}}
}
\]
Passing to cohomology now leads to the following commutative diagram:
\[
\xymatrix{
H^3(\operatorname{Cone}(t_Y);H_2(\widetilde{X}))\ar[r]^-{F_*}\ar[d]_\cong^-{\phi_{t,\nu}^*}
&H^3(\operatorname{Cone}(t_{Y});\Res_uH_2(\widetilde{X}'))\ar[d]_\cong^-{\phi_{t,\nu}^*}
&H^3(\operatorname{Cone}((t_u')_{Y'});\Res_uH_2(\widetilde{X}')) \ar[l]_-{(\id,h)^*}\ar[d]_\cong^-{\phi_{t,\nu}^*}\\
H^3(M(\nu|_Y),Y;H_2(\widetilde{X}))\ar[r]^-{F_*}
&H^3(M(\nu|_Y),Y;\Res_uH_2(\widetilde{X}'))
&H^3(M(\nu'|_{Y'}),Y';H_2(\widetilde{X}')).\ar[l]_-{(u,h)^*}
}
\]
Here, in the bottom right,  we used Lemma~\ref{lem:IgnoreRes} to make the $k$-invariant preserving identification
\begin{align*}
H^3(M(\nu'|_{Y'}),Y';H_2(\widetilde{X}'))
&=H^3(\operatorname{Cone}((\widetilde{\nu}'|_{Y'})_*);H_2(\widetilde{X}'))
=H^3(\Res_u \operatorname{Cone}((\widetilde{\nu}'|_{Y'})_*); \Res_u H_2(\widetilde{X}')) \\
&=H^3(\Res_u C(\widetilde{M}(\nu'|_{Y'}),Y'); \Res_u H_2(\widetilde{X}')).
\end{align*}
The claim now follows from the commutativity of this diagram
\end{proof}
At this point, we are ready to prove the equivalence asserted by the theorem.
To do so, we note that the existence of a map~$g \colon P_2(X) \to P_2(X')$ as in the theorem statement is independent of the choice of the models for~$P_2(X)$ and~$P_2(X')$.
We therefore assume without loss of generality that~$P_2(X)$ and~$P_2(X')$ are respectively obtained from~$X$ and~$X'$ by attaching cells of dimension~$\geq 4$.

We now prove that the existence of the map~$g$ implies the condition on the relative~$k$-invariants.
If~$g\colon P_2(X) \to P_2(X')$ induces~$u$ and~$F$ on~$\pi_1$ and~$\pi_2$, respectively, and is such that the compositions~$Y\hookrightarrow X\to P_2(X)\xrightarrow{g} P_2(X')$ and~$Y\xrightarrow{h}Y'\hookrightarrow X'\to P_2(X')$ are homotopic, then (thanks to our choice of models for $P_2(X)$ and $P_2(X')$), the chain map $g_* \colon C_*(P_2(X))|_{[0,3]} \to \Res_uC_*(P_2(X'))|_{[0,3]}$ induces a chain map~$f:=g_* \colon C_*(\widetilde{X})|_{[0,3]} \to \Res_uC_*(\widetilde{X}')|_{[0,3]}$ that satisfies the conditions of Proposition~\ref{prop:kConjChain},  whence by the claim
$$(u,h)^*(k_{X',Y'}^{\nu'})=F_*(k_{X,Y}^\nu).$$
For the converse, we assume that~$(u,h)^*(k_{X',Y'}^{\nu'})=F_*(k_{X,Y}^\nu)$ and construct the required map~$g$.
Thanks to the claim and Proposition~\ref{prop:kConjChain}, the condition on the $k$-invariants produces a chain map~$f$, with the aforementioned properties.
Apply~\cite[Theorem 16]{WhiteheadCombinatorial2} to realise this~$f$ by a map~$g' \colon X^{(3)} \to (X')^{(3)}$ that induces~$u$
and~$F$ and whose restriction to~$Y^{(2)}$ satisfies~$g'_* = h_*$ on~$C_*(\widetilde{Y})|_{[0,2]}.$
Postcompose~$g'$ with the map~$(X')^{(3)} \to P_2(X')$ to get a map~$X^{(3)} \to P_2(X')$.
Since~$\pi_k(P_2(X'))=0$ for~$k \geq 3$,  one can extend~$g'$ further to 
a map~$g\colon P_2(X) \to P_2(X')$.
Now
$$Y\hookrightarrow X\to P_2(X)\xrightarrow{g} P_2(X')
 \quad \text{and} \quad Y \xrightarrow{h} Y'\hookrightarrow X' \to P_2(X')$$
  induce chain homotopic maps on~$C_*(\widetilde{Y})|_{[0,2]}.$
By~\cite[Theorem 14]{WhiteheadCombinatorial2} this implies that the restrictions of the compositions to~$Y^{(2)}$ are homotopic.
But now, since~$\pi_k(P_2(X'))=0$ for~$k\geq 3$ this means that they are homotopic as maps~$Y \to P_2(X')$. Thus~$g$ is a map as required.
\end{proof}

\section{Relative~$k$-invariants of spaces}
\label{sec:kInvariantNotCW}
This section defines relative~$k$-invariants for pairs of spaces that are homotopy equivalent to a~CW pair and then proves Theorem~\ref{thm:RealiseAlgebraic3Type-intro} from the introduction.

\medbreak

We begin by generalising the definition of the map $(u,h)$ from Lemma~\ref{lem:InducedMapkInvariants} to the case where the target of $h$ is not a CW complex.

\begin{lemma}
\label{lem:InducedMapkInvariantsNotCW}
Let~$Y$ be a CW complex, let~$Y'$ be a space, and let~$\nu \colon Y\to B\pi$ and~$\nu'\colon Y'\to B\pi'$ be maps with $\nu$ cellular.
Then, up to homotopy rel.~$Y$,  a group isomorphism $u \colon \pi \to \pi'$ and a 
map~$h \colon Y \to Y'$ satisfying~$u\circ \nu_*=\nu'_*\circ h_*$ give rise to a unique map
$$ (u,h) \colon M(\nu) \to M(\nu')$$	
that extends $h$ and induces $u$ on fundamental groups.
	In particular, for any $\Z[\pi']$-module $A$,  the pair $(u,h)$ induces a group homomorphism 
	$$(u,h)^*\colon H^3(M(\nu'),Y';A)\to H^3(M(\nu),Y;\Res_u A).$$
	When $Y,Y'$ are CW complexes, and $\nu,\nu'$ are cellular,  the chain map induced by~$(u,h)$ is homotopic to the one described in Lemma~\ref{lem:InducedMapkInvariants}.
\end{lemma}
\begin{proof}
	This follows directly from the property that (based) maps into $B\pi'$ are determined up to (based) homotopy by the induced map on fundamental groups and that $M(\nu')$ is again a model for $B\pi'$. We give more details for the reader's convenience.
	
Realise $u \colon \pi \to \pi'$ by a map $v \colon  B\pi \to B\pi'$, so that $v \circ \nu \simeq \nu' \circ h$.
It follows that $Y\xrightarrow{h}Y'\hookrightarrow M(\nu')$ is homotopic to~$g\colon Y\xrightarrow{\nu}B\pi\xrightarrow{v}B\pi'\hookrightarrow M(\nu')$.
Choosing a homotopy~$H \colon Y \times I \to M(\nu')$ between these maps leads to a map extending $h$ that induces $u$ on fundamental groups, namely 
$$(u,h):=g\cup H\colon 
M(\nu)\to M(\nu').$$
 It remains to show uniqueness.	
	Let $g,g'\colon M(\nu)\to M(\nu')$ be maps that extend $h$ and induce $u$ on fundamental groups
and let $H,H'$ be defined defined as above.
We show that~$(u,h):=g \cup H$ and~$(u,h)':=g' \cup H'$ are homotopic rel.~$Y$.
Since~$M(\nu')$ is~$3$-coconnected,  and the restriction~$(u,h)|_Y=g|_Y=h=g'|_Y=(u,h)'|_Y \colon Y \to Y' \subset M(\nu')$ is a homotopy equivalence,~\cite[Lemma 5.3]{ConwayKasprowski4Manifolds}
ensures that there is a homotopy equivalence~$\phi \colon M(\nu') \to M(\nu')$ that is the identity on~$Y'$ and such that~$(u,h)' \simeq \phi \circ (u,h)$.

When both $Y$ and $Y'$ are CW complexes and $\nu$ and $\nu'$ are cellular, it follows from Lemma~\ref{lem:Cofibre} that $(u,h)_*$ is chain homotopic to the map from Lemma~\ref{lem:InducedMapkInvariants}.
\end{proof}

For a pair $(X,Y)$ that is homotopy equivalent to a CW pair, the following lemma will make it possible to define $k_{X,Y}$ by pulling back the relative $k$-invariant of any CW pair that is homotopy equivalent to $(X,Y)$.

\begin{lemma}
	\label{lem:k-inv-well-defined}
Let $(X,Y)$ be a pair of spaces,  let $(f_i,h_i) \colon (X_i',Y_i') \to (X,Y)$ be a homotopy equivalence with  $(X_i',Y_i')$ a CW pair, and set $u_i:=(f_i)_*$ for $i=0,1$.
Fix a map $\nu \colon X \to B\pi_1(X)$ and a cellular map $\nu_i' \colon X_i' \to B\pi_1(X_i')$ for $i=0.1$.
For $i=0,1$, the composition
\[H^3(M(\nu_i'),Y_i';\pi_2(X_i'))\xleftarrow{(u_i,h_i)^*,\cong}H^3(M(\nu),Y;\Res_{u_i^{-1}}\pi_2(X'_i))\xrightarrow{(f_i)_*,\cong}H^3(M(\nu),Y;\pi_2(X))\]
 takes $k_{X_0',Y_0'}^{\nu_0'}$ and $k_{X_1',Y_1'}^{\nu_1'}$ to the same element in $H^3(M(\nu),Y;\pi_2(X))$.
\end{lemma}
\begin{proof}
Note that our choice of model for $B\pi_1(X)$ and $B\pi_1(X_i')$ ensures that on fundamental groups, ~$(\nu|_Y)_* \circ (h_i)_*= u_i \circ (\nu_i'|_{Y_i'})_*$ for $i=0,1$.
Consequently, the maps $(u_i,h_i)$ from Lemma~\ref{lem:InducedMapkInvariantsNotCW} are defined.
Consider the following diagram:
\[
\xymatrix{
H^3(M(\nu_0'|_{Y_0'}),Y_0';\pi_2(X_0'))\ar[d]_-{(f_0)_*}&\ar[l]_-{(u_0,h_0)^*}H^3(M(\nu|_Y),Y;\Res_{u_0^{-1}}\pi_2(X_0'))\ar[rdd]^{(f_0)_*}\ar[d]^{(f_0)_*}& \\
H^3(M(\nu_0'|_{Y_0'}),Y_0';\Res_{u_0}\pi_2(X))\ar[d]_{(f_1)_*^{-1}}&H^3(M(\nu|_Y),Y;\pi_2(X))\ar[d]^{(f_1)_*^{-1}} & \\
H^3(M(\nu_0'|_{Y_0'}),Y_0';\Res_{u_1^{-1} \circ u_0}\pi_2(X_1'))&\ar[l]_-{(u_0,h_0)^*}
H^3(M(\nu|_Y),Y,\Res_{u_1^{-1}}\pi_2(X_1'))\ar[r]^-{(f_1)_*}& 
H^3(M(\nu|_Y),Y;\pi_2(X)). \\
H^3(M(\nu|_Y),Y;\Res_{u_1^{-1}}\pi_2(X_1'))\ar[d]_-{(u_1,h_1)^*}\ar[u]^-{(u_0,h_0)^*}\ar[dr]^{=}&& \\
H^3(M(\nu_1'|_{Y_1'}),Y_1';\pi_2(X_1'))&\ar[l]_-{(u_1,h_1)^*}H^3(M(\nu|_Y),Y;\Res_{u_1^{-1}}\pi_2(X_1'))\ar[ruu]_{(f_1)_*}\ar[uu]^-=&
}
\]
The commutativity of the triangles and of the lower square are immediate.
The commutativity of the top square follows from the naturality of the maps from Lemma~\ref{lem:InducedMapkInvariantsNotCW}.

Choose homotopy inverses $\overline{f}_1$ of $f_1$ and $\overline{h}_1$ of $h_1$ respectively. 
Set $f':=\overline{f}_1 \circ f_0,h':=\overline{h}_1 \circ h_0$ and $u':=u_1^{-1} \circ u_0$.
Note that $(u',h')^* \circ (u_1,h_1)^*=(u_0,h_0)^*$ and $(f')_*=(f_1)^{-1}\circ (f_0)_*$, whence
\begin{align*} (f_0)_* \circ (u_0,h_0)^{-*}(k_{X_0',Y_0'}^{\nu_0'})
&=(f_1)_* \circ (u_1, h_1)^{-*}\circ(u_1,h_1)^*\circ(u_0,h_0)^{-*}\circ(f_1)_*^{-1}\circ(f_0)_* (k_{X_0',Y_0'}^{\nu_0'}) \\
&=(f_1)_* \circ (u_1, h_1)^{-*}\circ(u',h')^{-*}\circ(f')_*(k_{X_0',Y_0'}^{\nu_0'}) \\
&=(f_1)_* \circ (u_1, h_1)^{-*}(k_{X_1',Y_1'}^{\nu_1'}).
\end{align*}
Here,  the first equality uses the commutativity of the diagram above, the second follows from the definition of $f',h',u'$,  and the last relies on Theorem~\ref{thm:GoalRelativekCW}.
\end{proof}

We can now define the $k$-invariant of a pair of spaces that is homotopy equivalent to a CW pair. \cref{lem:k-inv-well-defined} ensures that this is independent of the choice of homotopy equivalence.

\begin{notation}
\label{not:kinvNotCW}
Let~$(X,Y)$ be a pair of spaces that is homotopic to a CW pair,  say via a homotopy equivalence $(f,g)\colon (X',Y')\to (X,Y)$ with~$(X',Y')$ a CW pair. 
Set~$u:=(f_*)\colon \pi_1(X')\to\pi_1(X)$.
Given a map~$\nu \colon X \to B\pi_1(X)$ and a cellular map~$\nu' \colon X' \to B\pi_1(X')$,
define the \emph{relative~$k$-invariant}
$$k_{X,Y}^\nu \in H^3(M(\nu|_Y),Y;\pi_2(X))$$
 to be the image of~$k_{X',Y'}^{\nu'}$ under the composition
	\[H^3(M(\nu'|_{Y'}),Y';\pi_2(X'))\xleftarrow{(u,g)^*,\cong}H^3(M(\nu|_Y),Y;\pi_2(X'))\xrightarrow{f_*,\cong}H^3(M(\nu|_Y),Y;\pi_2(X)).\]
\end{notation}

\begin{remark}
\label{rem:Indepj}
Given another~$\widetilde{\nu} \colon X \to B\pi_1(X)$, there is a $k$-invariant preserving isomorphism
$$H^3(M(\widetilde{\nu}|_Y),Y;\pi_2(X)) \to H^3(M(\nu|_Y),Y;\pi_2(X)),$$
namely the unique dashed map that makes the following diagram commute
\[
\xymatrix{
H^3(M(\nu'|_{Y'}),Y';\pi_2(X'))\ar[r]^{(u,g)^{-*}}_\cong \ar[d]^{(\id_\pi,\id_Y)^*}_\cong
&H^3(M(\nu|_Y),Y;\pi_2(X'))\ar[r]^{f_*}_\cong
&H^3(M(\nu|_Y),Y;\pi_2(X)) \ar@{-->}[d]\\	
H^3(M(\widetilde{\nu}'|_{Y'}),Y';\pi_2(X'))\ar[r]^{(u,g)^{-*}}_\cong 
&H^3(M(\widetilde{\nu}|_Y),Y;\pi_2(X'))\ar[r]^{f_*}_\cong
&H^3(M(\widetilde{\nu}|_Y),Y;\pi_2(X)).
}
\]
\end{remark}
This remark leads to the following definition.

\begin{definition}
\label{def:kinvNotCW}
The \emph{relative $k$-invariant} of a pair~$(X,Y)$ that is homotopy equivalent to a~CW pair refers to 
$$k_{X,Y}^\nu \in  H^3(M(\nu|_Y),Y;\pi_2(X))$$
for any choice of~$\nu \colon X \to B\pi_1(X)$.
\end{definition}

We defer a discussion of the invariance of $k_{X,Y}^\nu$ to Remark~\ref{rem:Invariance}. Instead,  we begin with a generalisation  of \cref{thm:GoalRelativekCW} from CW pairs to pairs that are homotopy equivalent to~CW pairs.
\begin{theorem}
	\label{thm:GoalRelativek-spaces}
	Let~$(X_0,Y_0)$ and~$(X_1,Y_1)$ be pairs of spaces that are homotopy equivalent to CW pairs.
	For an isomorphism~$u\colon \pi_1(X_0)\to \pi_1(X_1)$, a map~$h \colon Y_0 \to Y_1$ with $u\circ(\iota_0)_*= (\iota_{Y_1})_* \circ h$ and a $u$-equivariant homomorphism~$F \colon \pi_2(X_0) \to \pi_2(X_1)$,  the following assertions are equivalent:
	\begin{itemize}
		\item there is a map~$g \colon P_2(X_0) \to P_2(X_1)$ between the Postnikov $2$-types that induces $u$ and $F$ on $\pi_1$ and~$\pi_2$, respectively, and such that the following maps are homotopic
		 \[Y_0\hookrightarrow X_0\to P_2(X_0)\xrightarrow{g}P_2(X_1)\qquad\text{and}\qquad Y_1\xrightarrow{h}Y_1\hookrightarrow X_1\to P_2(X_1);\]
		\item the relative $k$-invariants satisfy
		$$(u,h)^*(k_{X_1,Y_1}^{\nu_1})=F_*(k_{X_0,Y_0}^{\nu_0}) \in H^3(M(\nu_0|_{Y_0}),Y_0;\Res_u \pi_2(X_1)),$$
		for every $\nu_0 \colon X_0 \to B\pi_1(X_0)$ and $\nu_1 \colon X_1 \to B\pi_1(X_1)$.
	\end{itemize}
\end{theorem}
\begin{proof}
For $i=0,1$, choose a homotopy equivalence~$(f_i,g_i)\colon (X_i',Y_i')\to (X_i,Y_i)$ with~$(X_i',Y_i')$ a CW pair. The homotopy equivalences~$f_i$ induce homotopy equivalences~$P_2(f_i)\colon P_2(X_i')\to P_2(X_i)$. 
For $i=0,1$, fix a map~$\nu_i \colon X_i \to B\pi_1(X_i)$ and a cellular map~$\nu_i' \colon X_i' \to B\pi_1(X_i')$.
Additionally, choose a homotopy inverse $g_0^{-1}$ of $g_0$,  and consider 
\begin{align*}
u'&:=u_1^{-1} \circ u\circ u_0\colon \pi_1(X_0')\to \pi_1(X_1'), \\
h'&:=g_1^{-1}\circ h\circ g_0, \colon Y_0' \to Y_1',   \\
F'&:=(f_1)_*^{-1}\circ F\circ (f_0)_*\colon \pi_2(X_0')\to \pi_2(X_1').
\end{align*}
\begin{claim}
The following equality holds
	\[(u,h)^*(k_{X_1,Y_1}^{\nu_1})=F_*(k_{X_0,Y_0}^{\nu_0}) \in H^3(M(\nu_0|_{Y_0}),Y_0;\Res_u \pi_2(X_1))\]
if and only if the following equality holds:
	\[(u',h')^*(k_{X'_1,Y'_1}^{\nu_1'})=F'_*(k_{X_0',Y_0'}^{\nu_0'}) \in H^3(M(\nu_0'|_{Y_0'}),Y'_0;\Res_{u'}\pi_2(X'_1)).\]
	\end{claim}
	\begin{proof}
	Consider the following diagram:
	\[
\xymatrix{
H^3(M(\nu_0'|_{Y_0'}),Y_0';\pi_2(X_0')) \ar[d]_-{F'_*}
&H^3(M(\nu_0|_{Y_0}),Y_0;\Res_{u_0^{-1}}\pi_2(X_0'))  \ar[d]_-{F'_*} \ar[r]^-{(f_0)_*}   \ar[l]_-{(u_0,g_0)^*}
&H^3(M(\nu_0|_{Y_0}),Y_0;\pi_2(X_0))   \ar[d]_-{F_*}\\
H^3(M(\nu_0'|_{Y_0'}),Y_0';\Res_{u'}\pi_2(X_0'))
&H^3(M(\nu_0|_{Y_0}),Y_0;\Res_{u_1^{-1} \circ u}\pi_2(X_1'))  \ar[l]_-{(u_0,g_0)^*} \ar[r]^-{(f_1)_*}
&H^3(M(\nu_0|_{Y_0}),Y_0;\Res_u\pi_2(X_1)) \\
H^3(M(\nu_1'|_{Y_1'}),Y_1';\pi_2(X_1')) \ar[u]^-{(u',h')^*}
&H^3(M(\nu_1|_{Y_1}),Y_1;\Res_{u_1^{-1}}\pi_2(X_1'))  \ar[l]_-{(u_1,g_1)^*}\ar[r]^-{(f_1)_*} \ar[u]^-{(u,h)^*}
&H^3(M(\nu_1|_{Y_1}),Y_1;\pi_2(X_1)). \ar[u]^-{(u,h)^*}
}
\]
The upper left square commutes by naturality of $(u_0,g_0)^*$, the upper right square commutes by definition of $F'$, the lower right square commutes by naturality of $(u,h)^*$ and the lower left square can be seen to commute using the uniqueness properties of Lemma~\ref{lem:Cofibre}.
The claim now follows from the commutativity of the entirety of this diagram together with the fact that all the arrows involved are isomorphisms.
	\end{proof}
	
	Thanks to the claim and to \cref{thm:GoalRelativekCW},  the equality $(u,h)^*(k_{X_1,Y_1}^{\nu_1})=F_*(k_{X_0,Y_0}^{\nu_0}) $ is equivalent to the equality~$(u',h')^*(k_{X'_1,Y'_1}^{\nu_1'})=F'_*(k_{X_0',Y_0'}^{\nu_0'})$ which is in turn equivalent to the existence of a map~$g' \colon P_2(X_0') \to P_2(X_1')$ that induces~$u'$ and~$F'$ and such that
$Y_0'\hookrightarrow X_0'\to P_2(X_0')\xrightarrow{g'} P_2(X_1')$ and $Y_0'\xrightarrow{h}Y_1'\hookrightarrow X_1'\to P_2(X_1')$ are homotopic.
Finally, this is equivalent to the existence of a map $g \colon P_2(X_0) \to P_2(X_1)$ that induces~$u$ and~$F$ and such that~$Y_0\to X_0\to P_2(X_0)\xrightarrow{g} P_2(X_1)$ is homotopic to~$Y_0\xrightarrow{h}Y_1\to X_1\to P_2(X_1)$, as can be seen by the following (homotopy) commutative diagram:
$$
\xymatrix{
Y_0 \ar[r]\ar@/_3pc/[ddd]^{h'}&X_0\ar[r]&P_2(X_0) \ar@{-->}@/^3pc/[ddd]_{g} \\
Y_0' \ar[u]^{g_0} \ar[d]_{h}\ar[r]&X_0'\ar[r]\ar[u]^{f_0}&P_2(X_0')\ar[u]^-{P_2(f_0)} \ar@{-->}@/^1pc/[d]_{g'}\\
Y_1' \ar[r] \ar[d]_{g_1}&X_1'\ar[r]\ar[d]_{f_1}&P_2(X_1')\ar[d]_-{P_2(f_1)} \\
Y_1\ar[r]&X_1\ar[r]&P_2(X_1).
}
$$
This completes the proof of the theorem.
\end{proof}

\begin{remark}
\label{rem:Invariance}
We briefly discuss the sense in which the relative~$k$-invariant is an invariant of the pair~$(X,Y)$ up to homotopy equivalences that restrict to the identity on $Y$.
If  $(X,Y)$ and $(X',Y)$ are pairs of spaces that are homotopy equivalent to CW pairs and~$f \colon X \to X'$ is a homotopy equivalence that restricts to the identity on $Y$,  then Theorem~\ref{thm:GoalRelativek-spaces} shows~$(f_*,\id)^*(k_{X',Y}^{\nu'})=f_*(k_{X,Y}^{\nu})$ for every $\nu\colon X\to B\pi_1(X)$ and $\nu'\colon X\to B\pi_1(X')$.
Thus, if one uses~$f_*$ to identify~$\pi_i(X)$ with~$\pi_i(X')$ for $i=1,2$, as is common for the usual (absolute)~$k$-invariant,  and chooses $\nu$ and $\nu'$ such that~$\nu|_Y=\nu'|_Y$,  then the relative~$k$-invariants of~$(X,Y)$ and~$(X',Y)$ agree.
\end{remark}

We can now prove Theorem~\ref{thm:RealiseAlgebraic3Type-intro} from the introduction.

\begin{customthm}{\ref{thm:RealiseAlgebraic3Type-intro}}
	Let~$(X_0,Y_0)$ and~$(X_1,Y_1)$ be pairs of spaces that are homotopy equivalent to~CW pairs, and let $c_1\colon X_1\to P_2(X_1)$ be the Postnikov $2$-type of $X_1$.
For a map~$h \colon Y_0 \to Y_1$, an isomorphism~$u\colon \pi_1(X_0)\to \pi_1(X_1)$ with $u\circ(\iota_0)_*= (\iota_1)_* \circ h$, and a $u$-equivariant homomorphism~$F \colon \pi_2(X_0) \to~\pi_2(X_1)$, the following assertions are equivalent:
	\begin{itemize}
		\item there is a map~$c_0 \colon X_0\to P_2(X_1)$ such that 
		\[(c_0)_*=u,
		\quad c_1 \circ h\simeq c_0|_{Y_0},  \quad \text{and} \quad 
(c_0)_*=F
		;\]
		\item the relative $k$-invariants satisfy
		$$(u,h)^*(k_{X_1,Y_1}^{\nu_1})=F_*(k_{X_0,Y_0}^{\nu_0}) \in H^3(M(\nu_0|_{Y_0}),Y_0;\Res_u \pi_2(X_1))$$
for every $\nu_0 \colon X_0 \to B\pi_1(X_0)$ and $\nu_1 \colon X_1 \to B\pi_1(X_1)$.
	\end{itemize}
	\end{customthm}
		\begin{proof}
Let $c_0' \colon X_0 \to P_2(X_0)$ be the Postnikov $2$-type of $X_0$, and fix a map~$\nu_i \colon X_i \to B\pi_1(X_i)$ for~$i=0,1$.
If the equality~$(u,h)^*(k_{X_1,Y_1}^{\nu_1})=F_*(k_{X_0,Y_0}^{\nu_0}) \in H^3(M(\nu_0|_{Y_0}),Y_0;\Res_u \pi_2(X_1))$ holds,  then by Theorem~\ref{thm:GoalRelativek-spaces} there is a map~$g \colon P_2(X_0) \to P_2(X_1)$ that satisfies~$c_1\circ h\simeq g\circ c_0'|_Y$, and induces~$u$ and~$F$ (meaning that~$(c_1)_*^{-1} \circ g_* \circ (c_0')_*=F$).
Setting~$c_0:=g \circ c_0'$,  we obtain the desired map.

We prove the converse.
Let~$c_0 \colon X_0 \to P_2(X_1)$ be as in the theorem. 
Then~$c_0$ factors through~$P_2(X_0)$, i.e.\ there is a~$g\colon P_2(X_0)\to P_2(X_1)$ with~$c_0\simeq g\circ c_0'$: indeed~$P_2(X_0)$ can be obtained from~$X_0$ by attaching cells of dimension~$n \geq 4$,  and~$\pi_k(P_2(X_1))=0$ for~$k\geq~3$.
			The map~$g$ still induces~$u$ on $\pi_1$, and~$F$ on $\pi_2$. Furthermore,~$g\circ c_0'|_{Y_0}\simeq c_0|_{Y_0} \simeq c_1\circ h$.
			Again by Theorem~\ref{thm:GoalRelativek-spaces},  it follows that~$(u,h)^*(k_{X_1,Y_1}^{\nu_1})=F_*(k_{X_0,Y_0}^{\nu_0}) \in H^3(M(\nu_0|_{Y_0}),Y_0;\Res_u\pi_2(X_1))$.
		\end{proof}
	
We record the following immediate consequence that will be of use in~\cite{ConwayKasprowski4Manifolds}.
	
	\begin{corollary}
		Let $(X,Y)$ be a pair of spaces that is homotopy equivalent to a CW pair, and write~$i \colon Y\to X$ for the inclusion. 
Then,  for every $\nu \colon X \to B\pi_1(X)$,  the following equality holds:
$$(\id,i)^*k_{X,X}^\nu=k_{X,Y}^{\nu}\in H^3(M(\nu|_Y),Y;\pi_2(X)).$$
	\end{corollary}
	\begin{proof}
Let $c\colon X\to P_2(X)$ be the Postnikov $2$-type of $X$.
The corollary is obtained by applying \cref{thm:RealiseAlgebraic3Type-intro} to the pairs~$(X_0,Y_0)=(X,Y)$, $(X_1,Y_1)=(X,X)$,  with~$u=\id_{\pi_1(X)}$, $h=i$, $F=\id_{\pi_2(X)}$, and~$c_0=c=c_1$.
	\end{proof}
	
	As mentioned in the introduction, the following corollary shows that the relative $k$-invariant is indeed the obstruction for a section~$B\pi_1(X)\to P_2(X)$ extending~$Y \hookrightarrow X \to P_2(X)$.
	\begin{corollary}
		\label{lem:k-obstruction}
		Let $(X,Y)$ be a pair of spaces that is homotopy equivalent to a CW pair,  and let~$c\colon X\to P_2(X)$ be the Postnikov $2$-type of $X$.
For any choice of $\nu \colon X \to B\pi_1(X)$, the relative~$k$-invariant~$k_{X,Y}^\nu$ vanishes  if and only if the dashed map in the following diagram exists and both triangles commute up to homotopy:
		\[\begin{tikzcd}
			Y\ar[r,"c|_Y"]\ar[d,"\nu|_Y"']&P_2(X)\ar[d]\\
			B\pi_1(X)\ar[r,"\id"]\ar[ur,dashed]&B\pi_1(X).
		\end{tikzcd}\]
	\end{corollary}
	\begin{proof}
Apply \cref{thm:RealiseAlgebraic3Type-intro} with~$(X_1,Y_1)=(X,Y)$,~$(X_0,Y_0)=(M(\nu|_Y),Y)$,~$u=\id_{\pi_1(X)}$,~$h=\id_Y$, and~$F=0$.
		If follows that $k_{X,Y}^\nu$ vanishes if and only if there exists a map~$c_0 \colon M(\nu|_Y)\to P_2(X)$ such that $c_0|_Y\simeq c|_Y$ and~$(c_0)_*=\id_{\pi_1(X)}$. 
These are precisely the conditions that the two triangles in the diagram commute up to homotopy but with~$M(\nu|_Y)$ in place of~$B\pi_1(X)$.
Since $\nu|_Y$ factors as~$Y \to B\pi_1(X) \simeq M(\nu|_Y)$, the result follows.
	\end{proof}

\section{Additional properties of relative $k$-invariants.}
\label{sec:AdditionalProperties}

This final section collects additional properties of relative $k$-invariants that we require in~\cite{ConwayKasprowski4Manifolds}.

\begin{lemma}
\label{lem:kInvariantGoesToZeroPair}
Let $(X,Y)$ be a pair of spaces that is homotopy equivalent to a CW pair, and fix a map~$\nu \colon X \to B\pi_1(X)$.
Write~$j \colon X \to (X,Y)$ for the inclusion.
If the induced map~$\pi_1(Y) \to \pi_1(X)$ is surjective, then
$$ j_*(k_{X,Y}^\nu)=0 \in H^3(M(\nu|_Y),Y;H_2(\widetilde{X},\widetilde{Y})). $$
\end{lemma}
\begin{proof}
We first prove the result for CW complexes,  and then for pairs that are homotopy equivalent to CW complexes.
Let $(X,Y)$ be a CW pair,  choose an augmentation preserving chain map~$ t \colon C_*(\widetilde{X}) \to C^{\pi_1(X)}$ and write $t_Y$ for its restriction to~$C_*(\widetilde{Y})$.
Using
$$\theta_{C(\widetilde{X}),C(\widetilde{Y})}^t \colon C_3^\pi \oplus C_2(\widetilde{Y}) \to H_2(\widetilde{X})$$
to denote a chain representative of~$k_{C(\widetilde{X}),C(\widetilde{Y})}^t \in H^3(\operatorname{Cone}(t_Y),H_2(\widetilde{X}))$, we will show that~$j_*(\theta_{C(\widetilde{X}),C(\widetilde{Y})}^t)$ is a coboundary.
Write~$i \colon C_*(\widetilde{Y}) \to C_*(\widetilde{X})$ for the inclusion
and~$\alpha \colon C_*^\pi|_{[0,2]} \to C_*(\widetilde{X})|_{[0,2]}$ for the chain map underlying the definition of $\theta_{C(\widetilde{X}),C(\widetilde{Y})}^t$.
Recall from Lemma~\ref{lem:existence-alpha} that there is a chain homotopy~$\alpha \circ t \circ i \simeq i+\phi \colon C_*(\widetilde{Y})|_{[0,2]} \to C_*(\widetilde{X})|_{[0,2]}$, say via~$\{D_k \colon C_k(\widetilde{Y}) \to C_{k+1}(\widetilde{X}) \}_{k \leq 2}$.
In particular, we have
\begin{equation}
	\label{eq:kInvariantGoesToZeroPair}
	\alpha_2 \circ t_2 \circ i_2-i_2-\phi=d_3^X \circ D_2+D_1 \circ d_2^Y.
\end{equation}
Since~$\pi_1(Y)\to \pi_1(X)=:\pi$ is surjective, we can furthermore assume that~$X$ is obtained from~$Y$ by attaching cells of dimension~$\geq 2$. 
To see this, attach
$2$-cells to $Y$ to make $\pi_1(Y) \to \pi_1(X)$ a~$\pi_1$-isomorphism, 
then wedge on $S^2$'s to make it $2$-connected, and continue this process inductively to obtain a~CW complex $X'$ obtained from $Y$ by adding cells of dimension $\geq 2$ together with a homotopy equivalence $X' \to X$ that is the identity on $Y$.
This replacement is possible because the relative $k$-invariant is invariant under homotopy equivalences that are the identity on $Y$; recall Remark~\ref{rem:Invariance}.
Since $X^{(1)}=Y^{(1)}$, we have~$H_2(\widetilde{X},\widetilde{Y}^{(1)})= \frac{C_2(\widetilde{X},\widetilde{Y}^{(1)})}{\im(d_3^X)}$
and consider the diagram
$$
\xymatrix{
	C_3^\pi \oplus C_2(\widetilde{Y}) \ar[d]_{\bsm d_2^\pi& t_2 \circ i_2 \\ 0 & -d_2^Y\esm}
	\ar[r]^-{\bsm \alpha_2\circ d_3^\pi &\phi \esm}\ar@/^{2.5pc}/[rrr]^{\theta_{C(\widetilde{X}),C(\widetilde{Y})}^t}& 
	 Z_2(\widetilde{X})
	\ar[rr]^{\proj} &&H_2(\widetilde{X})
	\ar[r]^{j_*}& H_2(\widetilde{X},\widetilde{Y}) \\
	C_2^\pi \oplus C_1(\widetilde{Y})\ar@/_2pc/[rrrr]_{f}
	\ar[r]^{\bsm \alpha_2&D_1 \esm}& C_2(\widetilde{X})
	\ar[r]^-{j_*}& C_2(\widetilde{X},\widetilde{Y}^{(1)})
	\ar[r]^{\proj}& \frac{C_2(\widetilde{X},\widetilde{Y}^{(1)})}{\im(d_3^X)}
	\ar[r]^-{=}& H_2(\widetilde{X},\widetilde{Y}^{(1)}). 
	\ar[u]
}
$$
This diagram commutes: using~\eqref{eq:kInvariantGoesToZeroPair}, a pair $(a,b) \in C_3^\pi \oplus C_2(\widetilde{Y})$ satisfies 
\begin{align*}
	f \circ \begin{pmatrix} d_3^\pi &t_2 \circ i_2 \\ 0&-d_2^Y \end{pmatrix}\begin{pmatrix}
		a \\ b
	\end{pmatrix}
	&= j_*(\alpha_2 \circ d_3^\pi(a)+\alpha_2 \circ t_2 \circ i_2(b)-D_1 \circ d_2^Y(b)) \\
	&= j_*(\alpha_2 \circ d_3^\pi(a)+d_3^X \circ D_2(b)+i_2(b)+\phi(b)) \\
	&\equiv  j_*(\alpha_2 \circ d_3^\pi(a)+\phi(b)) \\
	&=j_* \circ \theta_{C(\widetilde{X}),C(\widetilde{Y})}^t\begin{pmatrix}
		a \\ b
	\end{pmatrix}.
\end{align*}
This proves that $j_*(\theta_{C(\widetilde{X}),C(\widetilde{Y})}^t)$ is a coboundary,  so that
$$j_*(k_{C(\widetilde{X}),C(\widetilde{Y})}^t)=0 \in H^3(\operatorname{Cone}(t_Y),H_2(\widetilde{X},\widetilde{Y})).$$
We can now conclude the proof of the lemma when $(X,Y)$ is  a CW pair.
Choosing a map $\nu \colon X \to B\pi_1(X)$,  the naturality of the map $\phi_{t,\nu}^*$ from Lemma~\ref{lem:Identifying-targets} implies that 
$$j_*(k_{X,Y}^\nu)=j_* \circ \phi_{t,\nu}^*(k_{C(\widetilde{X}),C(\widetilde{Y})}^t)=\phi_{t,\nu}^* \circ j_*(k_{C(\widetilde{X}),C(\widetilde{Y})}^t)=0.$$
Finally, we assume that $(X,Y)$ is a pair of spaces that is homotopy equivalent to a pair of CW complexes $(X',Y')$, say via $(f,g) \colon (X',Y') \to (X,Y)$.
Setting $u:=f_*$,  and using the naturality of the map~$(u,g)^*$ (from Lemma~\ref{lem:InducedMapkInvariantsNotCW}), one now verifies that $j_*'(k_{X',Y'}^{\nu'})=0$ implies~$j_*(k_{X,Y}^\nu)=0$.
This concludes the proof of the lemma.
\end{proof}

\begin{lemma}
	\label{lem:Xi}
	Let $X$ be a space,  let $\nu \colon X \to B\pi_1(X)$ be  a map, set $\pi:=\pi_1(X)$, and given a left~$\Z[\pi]$-module $A$,  consider the group homomorphism
\begin{align*}
\Phi\colon \Hom_{\Z[\pi]}(H_2(\widetilde{X}),A)&\to H^3(M(\nu),X;A) \\
	\varphi &\mapsto \varphi_*(k_{X,X}^\nu).
	\end{align*}
Then, writing $\delta$ for the connecting homomorphism in the exact sequence of the pair $(M(\nu),X)$, we have
$$\Phi\circ \ev=\delta\colon H^2(X;A)\to H^3(M(\nu),X;A).$$
\end{lemma}
\begin{proof}
We begin with the case where~$X$ is a CW complex and~$\nu$ is cellular.
In what follows, we consider~$C_*^\pi:=C_*(\widetilde{B}\pi)$ as a free resolution of~$\Z$,  pick a map $\nu \colon X \to B\pi$,  and consider the chain map~$t:= \nu_*   \colon C_*(\widetilde{X}) \to C_*^\pi$.
We deduce that the map~$\phi_{t,\nu}^{\varphi} \colon C_*(\widetilde{M}(\nu),\widetilde{X}) \to \operatorname{Cone}(t)_*$ from Lemma~\ref{lem:Identifying-targets} can be taken to be the identity.
It follows that~$k_{C(\widetilde{X}),C(\widetilde{X})}^t=k_{X,X}^\nu$.

	Use Lemma~\ref{lem:existence-alpha} to construct maps~$\alpha \colon C_*^\pi|_{[0,2]} \to C_*(\widetilde{X})|_{[0,2]}$ and $\phi \colon L_2 \to \ker(d_2^{X})$ that satisfy~$\alpha \circ t \simeq \id+\phi$ as maps on~$C_*(\widetilde{X})|_{[0,2]}.$
	As in the definition of~$k$-invariant,  consider the cocycle
	$$
	\theta_{C(\widetilde{X}),C(\widetilde{X})}^t \colon \operatorname{Cone}(t)_3=C_3^\pi\oplus C_2(\widetilde{X})\xrightarrow{(\alpha_2\circ d_3^\pi,\phi)}Z_2(\widetilde{X})\to H_2(\widetilde{X}).$$	
	We verify that~$\Phi \circ \ev=\delta$.
Since $\nu$ is cellular, we can identify $C_*(\widetilde{M}(\nu))$ 
with the algebraic mapping cylinder of the chain map~$t=\nu_*   \colon C_*(\widetilde{X}) \to C_*^\pi$,  say~$C_*(\widetilde{M}(\nu))=C_*^\pi \oplus C_*(\widetilde{X}) \oplus C_{*-1}(\widetilde{X}).$
Under this identification,  and recalling that $C_*(\widetilde{M}(\nu),\widetilde{X}) = \operatorname{Cone}(t)_*$ can be obtained from the algebraic mapping cylinder by modding out the middle summand, we see that the image of a class~$[\varphi_2 \colon C_2(\widetilde{X}) \to A] \in H^2(X;A)$ under the connecting map~$\delta \colon Z_2(X;A) \to Z_3(M(\nu),X;A)$ is
$$\delta(\varphi_2)=(0,-\varphi_2) \colon C_2(\widetilde{M}(\nu),\widetilde{X}) = \operatorname{Cone}(t)_2\to A.$$
Since~$\alpha \circ t \simeq \id+\phi$ as maps on~$C_*(\widetilde{X})|_{[0,2]}$,  we have~$[\varphi_2]=[\varphi_2 \circ \alpha_2 \circ t_2-\varphi_2 \circ \phi]$.
We also note that since~$\varphi_2 \in C^2(X,A)$ is a cocycle, the map $\varphi_2 \colon C_2(\widetilde{X}) \to A$ descends to $H_2(\widetilde{X})$.
It follows that
\begin{align*}
\delta([\varphi_2])
&=[(0,\varphi_2)]
=[(0,\varphi_2 \circ \alpha_2 \circ t_2-\varphi_2 \circ \phi)]
=[(0,\varphi_2 \circ \alpha_2 \circ t_2-\varphi_2 \circ \phi)]
+[\delta_3^{\operatorname{Cone}(t)}(\varphi_2 \circ \alpha_2,0)] \\
&=(\varphi_2 \circ \alpha_2 \circ d_3^\pi, -\varphi_2 \circ \phi)
=[\varphi_2 \circ \theta_{C(\widetilde{X}),C(\widetilde{X})}^t ].
\end{align*}
	The definition of 
	 $\Phi$ now yields
	\[\delta([\varphi_2])
	=[\varphi_2 \circ \theta_{C(\widetilde{X}),C(\widetilde{X})}^t ]
	=[\ev ([\varphi_2]) \circ \theta_{C(\widetilde{X}),C(\widetilde{X})}^t]
	=\Phi \circ \ev ([\varphi_2]).\]
We now assume that $X$ is a homotopy equivalent to a CW complex, say via a map $f \colon X' \to X$.
Set $u:=f_*$, choose a map~$\nu' \colon X' \to B\pi_1(X')$ and consider the following diagram:
$$
\xymatrix@C1.5cm{
\Hom(H_2(\widetilde{X}),A)\ar[r]^-{(f_*)^*,\cong}\ar@/_4pc/[dd]_{\Phi}& \Hom(H_2(\widetilde{X}')), \Res_u  A) \ar@/^4pc/[dd]^{\Phi'} \\
H^2(X;A) \ar[d]^-{\delta}  \ar[u]_-{\ev} \ar[r]^{f^*,\cong}& H^2(X';\Res_u A) \ar[d]^-{\delta'}  \ar[u]_-{\ev'}  \\
H^3(M(\nu),X;A) \ar[r]^{f^*,\cong}& H^3(M(\nu'),X';\Res_u A).
}
$$
Since $X'$ is a CW complex, $\Phi' \circ\ev' \circ=\delta'$ and, by naturality, the two rectangles commute.
A rapid verification shows that $f^* \circ \Phi =\Phi' \circ (f_*)^*$.
Since the $f^*$ are isomorphisms, a short diagram chase shows that this  implies $\Phi \circ \ev=\delta$, as required.
\end{proof}

\bibliographystyle{alpha}
\def\MR#1{}
\bibliography{BiblioHomotopyBoundary}
\end{document}